\documentclass[a4paper,11pt]{article}
\usepackage{hyperref}
\usepackage{geometry}
\usepackage{amssymb}
\usepackage{amsmath}
\usepackage{amsthm}
\usepackage{authblk}
\usepackage{paralist}
\usepackage{titlefoot}
\usepackage{xcolor}
\usepackage[numbers]{natbib} 

\newtheorem{theorem}{Theorem}[section]
\newtheorem*{theorem*}{Theorem}
\newtheorem{definition}[theorem]{Definition}
\newtheorem*{definition*}{Definition}
\newtheorem{proposition}[theorem]{Proposition}
\newtheorem{lemma}[theorem]{Lemma}
\newtheorem*{lemma*}{Lemma}

\newtheorem*{corollary*}{Corollary}
\newtheoremstyle{boldremark}
    {\dimexpr\topsep/2\relax} 
    {\dimexpr\topsep/2\relax} 
    {}          
    {}          
    {\bfseries} 
    {.}         
    {.5em}      
    {}          
\theoremstyle{boldremark}
\newtheorem{remark}[theorem]{Remark}

\newtheorem{assumption}[theorem]{Assumption}
\newtheorem*{assumption*}{Assumption}

\newcommand{\R}{\mathbb{R}}

\newcommand{\W}{\mathbb{W}}
\renewcommand{\P}{\mathbb{P}}
\newcommand{\E}{\mathbb{E}}
\newcommand{\A}{\mathbb{A}}

\newcommand{\cF}{\mathcal{F}}

\newcommand{\cH}{\mathcal{H}}

\newcommand{\cP}{\mathcal{P}}

\newcommand{\cL}{\mathcal{L}}

\renewcommand{\epsilon}{\varepsilon}

\newcommand{\dt}{\,\mathrm{d}t}
\newcommand{\ds}{\,\mathrm{d}s}
\newcommand{\dlambda}{\,\mathrm{d}\lambda}
\newcommand{\dgamma}{\,\mathrm{d}\gamma}
\newcommand{\dW}{\,\mathrm{d}W}

\numberwithin{equation}{section}
\allowdisplaybreaks

\begin{document}


\title{Peng's Maximum Principle for McKean-Vlasov Stochastic Differential Equations with Common Noise}
\date{\today}
	
\author[1]{Johan Benedikt Spille}
\author[1]{Wilhelm Stannat}
	
\affil[1]{Technische Universität Berlin, Berlin, Germany}
	
	
\maketitle

\unmarkedfntext{\textit{Mathematics Subject Classification (2020) --- 
93E20, 
49K45, 
49N80, 
60H30, 
60H10, 
49N15 
}}
	
\unmarkedfntext{\textit{Keywords and phrases --- stochastic optimal control, Peng maximum principle, McKean-Vlasov stochastic differential equation with common noise, mean-field SDE, conditional McKean-Vlasov backward SDE, adjoint calculus}}
	
\unmarkedfntext{\textit{Mail}: \textbullet \href{mailto:spille@math.tu-berlin.de}{spille@math.tu-berlin.de},\,\textbullet \href{mailto:stannat@math.tu-berlin.de}{stannat@math.tu-berlin.de}}

\begin{abstract}
    We study a stochastic optimal control problem for McKean–Vlasov stochastic differential equations (SDEs) with common noise, where the dynamics depend on the conditional law of the state.
    We derive a stochastic maximum principle of Peng type without imposing convexity assumptions on the control domain.
    In comparison to the standard McKean--Vlasov case, the maximum principle for the common noise case contains a third adjoint state, which is needed to dualize all second-order Lions derivatives in the Taylor expansion of the cost functional.
    The additional adjoint state, first introduced in \cite{spille2026}, is given by a conditional McKean--Vlasov backward SDE.
    All three adjoint states together allow for a complete linearization of all contributions in the second-order expansion, including interactions between conditionally independent copies of the first variational process.
    
    As part of our analysis, we also prove a general well-posedness result for conditional McKean--Vlasov backward SDEs.
\end{abstract}

\tableofcontents

\section{Introduction}
In this paper, we are concerned with the following optimization problem: Minimize over all admissible controls the cost functional
\begin{equation}\label{eq:cost}
    J(\alpha)
    = \mathbb{E}\left[
    \int_0^T f(t,X_t,\mu_t,\alpha_t)\dt
    + g(X_T,\mu_T)
    \right],
\end{equation}
where $f:[0,T]\times \R^d\times \cP_2(\R^d)\times U\to \R$ and $g:\R^d\times \cP_2(\R^d)\to \R$ are deterministic functions, subject to the state equation, given by the controlled McKean--Vlasov SDE with common noise
\begin{equation}\label{eq:state}
    \mathrm{d}X_t = A(t,X_t,\mu_t,\alpha_t)\dt 
          + B(t,X_t,\mu_t,\alpha_t)\dW^1_t+C(t,X_t,\mu_t,\alpha_t)\dW^0_t,
    \qquad X_0 = x_0,
\end{equation}
where $(A,B,C):[0,T]\times \R^d\times \cP_2(\R^d)\times U\to \R^d\times\R^{d\times d}\times\R^{d\times d}$ are deterministic and measurable, 
$W^i$ are independent $d$-dimensional Brownian motions,
$\mu_t := \mathcal{L}(X_t\mid \cF^0_t)$ denotes the regular version of the conditional distribution of $X_t$ given $\cF_t^0$, which is the filtration generated by $W^0$, $x_0\in L^2(\Omega;\R^d)$ is independent of $W^0$ and $W^1$ and $\cP_2(\R^d)$ denotes the $2$-Wasserstein space.
A control process $\alpha=(\alpha_t)_{0\le t\le T}$ taking values in a non-empty, not necessarily convex metric space $U$ is called admissible if it is progressively measurable. 
The set of all admissible controls is denoted by $\mathbb{A}$. \\
In the above setting, we will deduce a necessary optimality condition called Peng's maximum principle given in Theorem \ref{Theorem:maximum principle}. 
This type of maximum principle was first introduced by \citet{Pontryagin1962} for the deterministic case and later extended by  \citet{Bismut1978} and \citet{Peng1990} to a stochastic setting. These foundational contributions establish adjoint equations and variational techniques that allow for the characterization of optimal controls in terms of forward--backward systems.

More recently, there has been increasing interest in control problems where the dynamics depend not only on the state but also on its distribution. Such SDEs, commonly referred to as McKean--Vlasov or mean-field type SDEs, arise as limits of interacting particle systems \cite{Kac1956,McKean1967,Sznitman1991} and in the limit of multi-player games, which gives rise to mean-field game theory as a framework for analyzing strategic interactions among a large number of agents. Pioneering contributions by \citet{LasryLions2007} and independently by \citet{HuangMalhamCaines2006MeanFieldGames} initiated a vast literature on control of such systems.

An additional layer of complexity arises when common noise is introduced \cite{CarmonaDelarueLacker2016MFGwithCommonNoise}. In this case, the system is influenced by both idiosyncratic and common sources of randomness, leading naturally to conditional McKean--Vlasov dynamics, where the law of the state is taken conditionally on the common noise filtration. This setting is particularly relevant in applications where systemic effects or aggregate uncertainties play a role, for instance in models of financial systems with systemic risk. The presence of common noise significantly alters the analytical structure, as conditional expectations replace unconditional ones and many classical arguments no longer apply.

In limits of mean-field games, one has to distinguish two different settings \cite{CarmonaDelarue2013MFGvsMFControl}. The first one is decentralized decision-making of players and the search of Nash equilibria, commonly referred to as control of mean-field games, where, in the limiting problem, one fixes first the distribution, then solves the control problem and afterwards finds the distribution that fits this solution, e.g. by fixed point theorems. On the other hand there are so-called mean-field control problems or control of McKean--Vlasov SDEs, which correspond to a centralized optimization perspective or cooperative games. Here, one has to optimize a McKean--Vlasov equation directly, which introduces the need of measure-dependent derivatives. Our problem lies in the mean-field control setting. Nevertheless, connections to mean-field games persist, e.g. control of (conditional) McKean--Vlasov equations appears in major--minor games \cite{CarmonaZhu2016MeanFieldsWithMajorMinorPlayers}. 

Several authors have developed stochastic maximum principles in the standard McKean--Vlasov setting. Results for stochastic Pontryagin maximum principles in McKean--Vlasov frameworks can be found in \cite{Andersson2011,CarmonaDelarue2015,CarmonaDelarue2018I,HocquetVogler2020}. The first Peng-type maximum principles were given by \citet{Buckdahn2011,Buckdahn2016,BuckdahnPeng2017}.

Results on mean-field control with common noise remain sparse. \citet{PhamWei2017DPPforCommonNoise} and \citet{Djete2022} have considered a dynamic programming principle approach, \citet{Pham2016LQR} and \citet{Bo2022LQRwithBrownianCommonNoise} treated a linear quadratic problem and \citet{bo2026extendedmeanfieldcontrolproblems} considered a maximum principle with Poissonian common noise. A different perspective is the control of a stochastic Fokker-Planck equation derived for the conditional law \cite{hambly2025optimalcontrolnonlinearstochastic}. Other than that, conditional McKean--Vlasov control problems in different settings have been considered, e.g. introduced through partial observations \cite{BuckdahnLiMa2017ControlWithPartialObservations} or the conditioning on the state leaving a domain \cite{carmona2025conditionalmckeanvlasovcontrol}. Worth mentioning is the Pontryagin maximum principle derived in the appendix in \cite{CarmonaZhu2016MeanFieldsWithMajorMinorPlayers}, which should immediately translate to the common noise setting and shows that for convex control domains, where only a first-order expansion of the cost functional has to be made, a generalization poses no difficulty.

The main challenge for a Peng-type maximum principle is the non-convex control domain, as the study of the variation of the cost functional has to be done via a second-order Taylor expansion.
This was first observed by \citet{Peng1990}. 
The main obstacle then lies in the correct dualization of the terms in the expansion of the cost functional, i.e. the identification of suitable adjoint states to replace terms, which depend implicitly on the variation.
In the McKean--Vlasov case, this expansion involves terms stemming from the Lions derivative, which include interactions between (conditionally) independent copies of the state and variational processes. 
For the classical McKean--Vlasov setting without common noise, it is possible to control these terms via refined estimates and show that they are of higher order (cf. \cite{Buckdahn2016} Proposition 4.3), implying that they play no role in the resulting maximum principle. However, in the presence of common noise this is no longer the case. 
The classical expectation becomes a conditional expectation and the independent copies become conditionally independent copies. Due to the lack of sufficient smoothing properties of conditional expectations, the same higher-order estimates are not possible (cf. Remark \ref{remark:example_ConditionalExpNotOfHigherOrder}).

The approach taken in this paper is designed to overcome these difficulties. 
Building on ideas from \citet{spille2026}, 
we introduce a third adjoint equation, given by the solution to a conditional McKean--Vlasov backward SDE, which allows for a dualization of all second-order terms appearing in the expansion of the cost functional, which cannot be dualized by the first two adjoint equations, as they involve multiple (conditionally) independent copies of the first-order variational process. 
After these dualizations, we are able to derive a maximum principle (cf. Theorem \ref{Theorem:maximum principle}). 
In contrast to the result for the classical McKean--Vlasov setting \cite{spille2026}, the third adjoint process does appear in the maximum principle. 
This shows that this result cannot be achieved without the introduction of the third adjoint equation and that the corresponding terms are not of higher-order in the common noise case.

We also compare this result to the dynamic programming approach from \citet{PhamWei2017DPPforCommonNoise}. In the non-McKean--Vlasov case, there is a canonical connection between the value function and the Hamiltonian maximum principle, i.e. one can identify the first and second adjoint processes as the derivatives of the value function. We heuristically see an analogous result in our setting (cf. Remark \ref{remark:relationToDPP}). The difference is that the second-order Lions derivative splits into two parts ($\partial_y\partial_\mu$ and $\partial_\mu\partial_\mu$), so the first adjoint corresponds to the first-order Lions derivative, the second adjoint corresponds to the $\partial_y\partial_\mu$-part of the second-order Lions derivative and the third adjoint corresponds to the $\partial_\mu\partial_\mu$-part of the second-order Lions derivative.

In the appendix, we give a needed second-order Taylor expansion and see that the remainder can be uniformly estimated if all second derivatives are uniformly Lipschitz. We further give a general well-posedness result for conditional McKean--Vlasov backward SDE.

The remainder of the paper is organized as follows. In Section \ref{sec:Preliminaries}, we introduce the necessary preliminaries, including Lions differentiability and the regularity assumptions on the coefficients. Section \ref{sec:The Variational Equations} is devoted to the definition and analysis of the first- and second-order variational equations. In Section \ref{sec:The Adjoint Equations}, we define the adjoint processes, including the third adjoint equation and its necessary setting. Section \ref{sec:The Duality Relations} establishes the key duality relations required for the variational analysis. In Section \ref{sec:Expansion of Cost Functional and Maximum Principle}, we combine these ingredients to derive the expansion of the cost functional, formulate the maximum principle and discuss its connection to the dynamic programming principle and the derived HJB-equation. In Section \ref{sec:An easy Example}, we illustrate the results with a simple example. Finally, in the Appendix \ref{appendix:Taylor formula for x and mu} we give the required Taylor formula for the expansion of the  cost functional and in Appendix \ref{appendix:Ex and Uni for Controlled MKVSDE with Common Noise} we give a well-posedness result for general conditional McKean--Vlasov backward SDE.

\section{Preliminaries}\label{sec:Preliminaries}
We will need the same assumptions as \cite{spille2026} for our coefficients, so we take these preliminaries from there.\\
We recall the notion of Lions differentiability. Let $\mathcal{P}_2(\mathbb{R}^d)$ denote the $2$-Wasserstein space (for more details, see \cite{Villani2009}).  
A mapping $\varphi:\mathcal{P}_2(\mathbb{R}^d)\to\mathbb{R}^m$ is said to be Lions differentiable at $\mu$ if there exists a measurable mapping 
$\partial_\mu\varphi(\mu):\mathbb{R}^d\to\mathbb{R}^{m\times d}$ such that, for any $Y\sim\mu$, the lifted mapping $\tilde{\varphi}: L^2(\Omega,\cF,\P;\R^d)\mapsto \R^{m},X\mapsto \varphi(\mathcal{L}(X))$ is Fréchet differentiable in $L^2(\Omega,\cF,\P;\R^d)$ and
\begin{equation}
    D\tilde\varphi(X) Y 
    = \mathbb{E}\left[\partial_\mu\varphi(\mu)(X) Y
    \right],
    \qquad \text{for all }Y\in L^2(\Omega,\cF,\P;\R^d).
\end{equation}
For more details we refer to \cite{CarmonaDelarue2018I}.
We now introduce the necessary regularity for the maximum principle (\cite{BuckdahnPeng2017,Buckdahn2016} from Definition 2.1, also \cite{CarmonaDelarue2018I} Chapter 5.6).
\begin{definition}
    We say that $\varphi \in C_b^{1,1}(\mathcal{P}_2(\mathbb{R}^d),\R^m)$,
    if there exists for all $\vartheta \in L^2(\Omega, \mathbb{R}^d)$ an $\cL(\vartheta)$-modification of $\partial_\mu \varphi(\cL(\vartheta))$, again denoted by $\partial_\mu \varphi(\cL(\vartheta))$, such that $\partial_\mu \varphi: \mathcal{P}_2(\mathbb{R}^d) \times \mathbb{R}^d \to \mathbb{R}^{m\times d}$ is bounded and Lipschitz continuous, i.e. there is some $C>0$ such that
    \begin{enumerate}[(i)]
        \item $|\partial_\mu \varphi(\mu)(y)| \leq C$, for all $ \mu \in \mathcal{P}_2(\mathbb{R}^d)$ and $ y \in \mathbb{R}^d$, and
        \item $|\partial_\mu \varphi(\mu)(y)-\partial_\mu \varphi(\mu^{\prime})(y^{\prime})| \leq C(W_2(\mu, \mu^{\prime})+|y-y^{\prime}|)$, for all $ \mu, \mu^{\prime} \in \mathcal{P}_2(\mathbb{R}^d)$ and $ y,y^{\prime} \in \mathbb{R}^d$.
    \end{enumerate}
\end{definition}
For $(\mu,y)\mapsto \partial_\mu\varphi(\mu)(y)$ we now have two different partial derivatives $\partial_{y}\partial_\mu \varphi$ and $\partial_{\mu}\partial_\mu \varphi$.
\begin{definition}
    We say that $\varphi \in C_b^{2,1}(\mathcal{P}_2(\mathbb{R}^d),\R^m)$, if $\varphi \in C_b^{1,1}(\mathcal{P}_2(\mathbb{R}^d),\R^m)$ and
    \begin{enumerate}[(i)]
        \item $\partial_\mu \varphi(\cdot)( y) \in C_b^{1,1}(\mathcal{P}_2(\mathbb{R}^d),\R^{m\times d})$, for all $y \in \mathbb{R}^d$, and $\partial_\mu^2 \varphi$ : $\mathcal{P}_2(\mathbb{R}^d) \times \mathbb{R}^d \times \mathbb{R}^d \to \mathbb{R}^{m\times d\times d}$ is bounded and Lipschitz-continuous and
        \item $\partial_\mu \varphi(\mu): \mathbb{R}^d \to \mathbb{R}^{m\times d}$ is differentiable, for every $\mu \in \mathcal{P}_2(\mathbb{R}^d)$, and its derivative $\partial_y \partial_\mu \varphi: \mathcal{P}_2(\mathbb{R}^d) \times \mathbb{R}^d \to \mathbb{R}^{m\times d\times d}$ is bounded and Lipschitz-continuous.
    \end{enumerate}
\end{definition}
These versions of the derivatives are unique (cf. \cite{CarmonaDelarue2018II} Remark 4.12) and we will always be using these versions.
For our coefficients we will need the following regularity.
\begin{definition}
    For $\varphi:\R^d\times\cP_2(\R^d)\to \R^m $ we say $\varphi\in C_b^{2,1}(\R^d\times\cP_2(\R^d),\R^m)$ if
    \begin{enumerate}[(i)]
        \item $\varphi(\cdot ,\mu)\in C^2_b(\R^d,\R^m)$ for all $\mu\in\cP_2(\R^d)$,
        \item $\varphi(x,\cdot)\in C^{2,1}_b(\cP_2(\R^d),\R^m)$ and $\partial_x\varphi(x,\cdot)\in C^{1,1}_b(\cP_2(\R^d),\R^{m\times d})$ for all $x\in\R^d$,
        \item $\partial_\mu\varphi(\cdot ,\mu)(\cdot)\in C^1_b(\R^d\times \R^d,\R^{m\times d})$ for all $\mu\in\cP_2(\R^d)$ and
        \item $\varphi$ and all first and second-order derivatives of $\varphi$ are bounded and Lipschitz continuous in all variables.
    \end{enumerate}
\end{definition}
We will use shorthand notations for our derivatives. For $\varphi \in C_b^{2,1}(\R^d\times\cP_2(\R^d),\R^m)$ define $\varphi_x:=\partial_x\varphi$, $\varphi_\mu:=\partial_\mu \varphi$, $\varphi_{xx}:=\partial_x\partial_x\varphi$. Further, for $(x,\mu,y)\mapsto \partial_\mu\varphi(x,\mu)(y)$, we denote $\varphi_{x\mu}:=\partial_{x}\partial_\mu \varphi$, $\varphi_{y\mu}:=\partial_{y}\partial_\mu \varphi$ and $\varphi_{\mu\mu}:=\partial_{\mu}\partial_\mu \varphi$.
\begin{assumption}\label{Assumption:C^{2,1}}
    The coefficients $A,B,C,f,g$ are measurable in all variables and for all $t\in[0,T]$ and $u\in U$ 
    \begin{enumerate}[(i)]
        \item $A(t,\cdot,\cdot,u)\in C_b^{2,1}(\R^d\times\cP_2(\R^d),\R^d)$,
        \item $B(t,\cdot,\cdot,u),C(t,\cdot,\cdot,u)\in C_b^{2,1}(\R^d\times\cP_2(\R^d),\R^{d\times d})$,
        \item $f(t,\cdot,\cdot,u)\in C_b^{2,1}(\R^d\times\cP_2(\R^d),\R)$ and 
        \item $g\in C_b^{2,1}(\R^d\times\cP_2(\R^d),\R)$,
    \end{enumerate}
    with Lipschitz and boundedness constants uniform in $t$ and $u$.
\end{assumption}
Remember that we define a control process $\alpha=(\alpha_t)_{0\le t\le T}$ taking values in the non-empty, not necessarily convex metric space $U$ to be admissible if it is progressively measurable and denote the set of all admissible controls by $\mathbb{A}$. Under the above assumptions there is a unique solution to the state equation \eqref{eq:state} for every $\alpha\in \A$. We discuss this in more detail in Appendix \ref{appendix:Ex and Uni for Controlled MKVSDE with Common Noise}.

For brevity's sake, we will often write $\theta_t:=(t,X_t,\mu_t,\alpha_t)$, where $\alpha$ will be optimal, $X$ the corresponding solution to \eqref{eq:state} and $\mu$ its conditional law. Later, $\theta$ will also contain the first adjoint states $(p,q,r)$ (which will be introduced in Section \ref{sec:The Adjoint Equations}) such that $\theta_t:=(t,X_t,\mu_t,\alpha_t,p_t,q_t,r_t)$. Generally, by an abuse of notation, we will also write $\theta$ into functions that takes less arguments, by which we imply that only the needed arguments are taken.

Moreover, we will often denote the transpose of a matrix $G\in\R^{d\times d}$ as its dual operator $G^*=G^\top$ and use the tensor product $x\otimes y:=xy^\top$ for $x,y\in \R^d$.

Further, we will be interested in the order of processes with respect to some parameter. For this we introduce the $\mathcal{O}$-notation: For mappings $f:(0,\infty)\to \R^m$, $g:(0,\infty)\to \R$, we write $f\in O(g(\epsilon))$ if $\limsup_{\epsilon\to 0} \left|\frac{f(\epsilon)}{g(\epsilon)}\right|< \infty$ and $f\in o(g(\epsilon))$ if $\lim_{\epsilon\to 0} \left|\frac{f(\epsilon)}{g(\epsilon)}\right|=0$.

\section{The Variational Equations}\label{sec:The Variational Equations}

For the rest of the paper, let $\alpha\in \A$ be optimal with corresponding state process $X$. Note that we do not make any statements about the existence of optimal controls in this paper, we refer to \cite{Djete2022EquivalenceBetweenDiffFormAndExistenceOfRelaxedOptimalControl} for some results on existence of optimal relaxed controls. For a measurable subset $E_\epsilon\subset [0,T]$ with Lebesgue measure $|E_\epsilon|=\epsilon>0$ and $\beta\in \A$ define the spike variation
\begin{equation*}
    \alpha_t^{\epsilon} \triangleq\begin{cases}
    \beta_t,&\text{for } t \in E_{\epsilon}, \\
    \alpha_t,&\text{for }  t \in E_{\epsilon}^c.
    \end{cases}
\end{equation*}
We denote $X^\epsilon$ the solution to the state equation \eqref{eq:state} corresponding to $\alpha^\epsilon$, $\mu^\epsilon$ the corresponding conditional distribution and $\Delta X = X^\epsilon-X$. Further, we denote $\delta A(t) = A(t,X_t,\mu_t,\alpha^\epsilon_t)-A(t,X_t,\mu_t,\alpha_t)$ and similarly for any other mapping. Due to the common noise, the variational equations will contain conditional laws. To write these down properly we need to introduce a lifting of the random variables. To simplify the notation, we work in a special setup: Let $W^i$ be $d$-dimensional Brownian motions on a complete filtered probability space $(\Omega^i,(\mathcal{F}^i)_{t\in[0,T]},\mathbb{P}^i)$. We assume that $(\mathcal{F}^0_t)_{0\le t\le T}$ is the augmented filtration generated by $W^0$ and 
$(\mathcal{F}^1_t)_{0\le t\le T}$ is the augmented filtration generated by $W^1$ and $x_0$ such that $x_0$ is independent of $W^1$. 
The whole analysis will now be done on the completed product space 
$(\Omega,(\mathcal{F})_{t\in[0,T]},\P):=\bigotimes_{i=0}^1 (\Omega^i,(\mathcal{F}^i)_{t\in[0,T]},\mathbb{P}^i)$, where we take the canonical lift of all random variables to the product space such that also $x_0$, $W^0$ and $W^1$ are independent.
Now, let $(\tilde{\Omega},(\tilde{\mathcal{F}})_{t\in[0,T]},\tilde{\P})=(\Omega^1,(\mathcal{F}^1)_{t\in[0,T]},\mathbb{P}^1)$ and, for a random variable $\Phi:\Omega\to \R^m$, define
\begin{equation*}
    \tilde{\Phi}:\Omega^0\times \tilde{\Omega}\to \R^m,\quad (\omega^0,\tilde{\omega}) \mapsto \Phi(\omega^0,\tilde{\omega}).
\end{equation*}
Then, for another random variable $\Psi:\Omega\to \R^m$ and for a measurable $\varphi$,
\begin{equation*}
    \tilde{\E}\left[\varphi(\Psi,\tilde{\Phi})\right]:\Omega \to \R^n ,\quad \omega=(\omega_0,\omega_1)\mapsto 
    \int_{\Omega^1}\varphi(\Psi(\omega),\Phi(\omega^0,\tilde{\omega}))\,\P^1(\mathrm{d}\tilde{\omega})
\end{equation*}
constitutes a measurable (by Fubini) random variable, which can be seen as a conditional expectation (cf. Lemma 2.4 \cite{CarmonaDelarue2018II}). We remember our $\theta$-notation from the end of Section \ref{sec:Preliminaries}.\\
Now, we define the first variational process $Y$ as a solution to
\begin{equation}
    \begin{aligned}
        \mathrm{d}Y_t 
        =& \left( A_x(\theta_t)Y_t +\tilde{\E}\left[A_\mu(\theta_t)( \tilde{X}_t) \tilde{Y}_t\right] 
        + \delta A(t) \right) \dt  \\
        &+ \left( B_x(\theta_t)Y_t + \tilde{\E}\left[B_\mu(\theta_t)( \tilde{X}_t) \tilde{Y}_t\right]
        + \delta B(t) \right) \dW^1_t \\
        &+ \left( C_x(\theta_t)Y_t + \tilde{\E}\left[C_\mu(\theta_t)( \tilde{X}_t) \tilde{Y}_t\right]
        + \delta C(t) \right) \dW^0_t  \\
        Y_0 =& 0,
    \end{aligned}\label{eq:first variational process SDE}
\end{equation}
and the second variational process $Z$ as a solution to
\begin{equation}
    \begin{aligned}
        \mathrm{d}Z_t 
        = & \Big( A_x(\theta_t)Z_t +\tilde{\E}\left[A_\mu(\theta_t)( \tilde{X}_t) \tilde{Z}_t\right] + \frac{1}{2} A_{xx}(\theta_t)[Y_t,Y_t]
        + \frac{1}{2}\tilde{\E}\left[ A_{y\mu}(\theta_t)( \tilde{X}_t)[ \tilde{Y}_t,\tilde{Y}_t]\right]\\
        & 
        + \frac{1}{2}\tilde{\E} \left[\tilde{\tilde\E}\left[ A_{\mu\mu}(\theta_t)( \tilde{X}_t,\tilde{\tilde{X}}_t)[ \tilde{Y}_t,\tilde{\tilde{Y}}_t]\right]\right]
        +\tilde{\E}\left[ A_{x\mu}(\theta_t)( \tilde{X}_t)[ \tilde{Y}_t,Y_t]\right]\\
        &+ \delta A_x(t) Y_t+\tilde{\E}\left[\delta A_\mu(t)(\tilde{X}_t) \tilde{Y}_t\right]  \Big)\dt  \\
        &+\Big( B_x(\theta_t)Z_t +\tilde{\E}\left[B_\mu(\theta_t)( \tilde{X}_t) \tilde{Z}_t\right] + \frac{1}{2} B_{xx}(\theta_t)[Y_t,Y_t]
        + \frac{1}{2}\tilde{\E}\left[ B_{y\mu}(\theta_t)( \tilde{X}_t)[ \tilde{Y}_t,\tilde{Y}_t]\right]\\
        & 
        + \frac{1}{2}\tilde{\E} \left[\tilde{\tilde\E}\left[ B_{\mu\mu}(\theta_t)( \tilde{X}_t,\tilde{\tilde{X}}_t)[ \tilde{Y}_t,\tilde{\tilde{Y}}_t]\right]\right]
        +\tilde{\E}\left[ B_{x\mu}(\theta_t)( \tilde{X}_t)[ \tilde{Y}_t,Y_t]\right]\\
        &+ \delta B_x(t) Y_t+\tilde{\E}\left[\delta B_\mu(t)(\tilde{X}_t) \tilde{Y}_t\right]\Big)  \dW^1_t\\
        &+\Big( C_x(\theta_t)Z_t +\tilde{\E}\left[C_\mu(\theta_t)( \tilde{X}_t) \tilde{Z}_t\right] + \frac{1}{2} C_{xx}(\theta_t)[Y_t,Y_t]
        + \frac{1}{2}\tilde{\E}\left[ C_{y\mu}(\theta_t)( \tilde{X}_t)[ \tilde{Y}_t,\tilde{Y}_t]\right]\\
        & 
        + \frac{1}{2}\tilde{\E} \left[\tilde{\tilde\E}\left[ C_{\mu\mu}(\theta_t)( \tilde{X}_t,\tilde{\tilde{X}}_t)[ \tilde{Y}_t,\tilde{\tilde{Y}}_t]\right]\right]
        +\tilde{\E}\left[ C_{x\mu}(\theta_t)( \tilde{X}_t)[ \tilde{Y}_t,Y_t]\right]\\
        &+ \delta C_x(t) Y_t+\tilde{\E}\left[\delta C_\mu(t)(\tilde{X}_t) \tilde{Y}_t\right]\Big)  \dW^0_t\\
        Z_0 =& 0. 
    \end{aligned}\label{eq:second variational process SDE}
\end{equation}
\begin{remark} \label{remark:after def of variational eq}
    \begin{enumerate}[(i)]
        \item The above equations do not only depend on $\cL(X_t\mid\cF^0_t)$, $\cL(Y_t\mid\cF^0_t)$ and $\cL(Z_t\mid\cF^0_t)$, but on the joint conditional law $\cL(X_t,Y_t\mid\cF^0_t)$ and $\cL(X_t,Z_t\mid\cF^0_t)$, but considering $(X,Y,Z)$ together in one equation, we get existence and uniqueness by \ref{theorem_appendix:ex and uni of MKV SDE w common noise}.
        \item Clearly, the coefficients of the above equations are only progressive with respect to the whole filtration due to the dependence on $X_t$ and $\alpha_t$, so one cannot get a closed form Fokker-Planck SPDE for the conditional law.
        \item We emphasize the dependence of $Y$ and $Z$ on the spike variation and therefore $\epsilon$, but refrain from denoting them $Y^\epsilon$ and $Z^\epsilon$ for the sake of readability.
        \item We note the measurability discussions in \cite{CarmonaDelarue2018II} after Theorem 4.14 and Theorem 4.17, which imply sufficient measurability such that all integrals are well-defined. Our assumptions correspond to so-called full $\mathcal{C}^2$-regularity therein.
        \item A standard argument shows that Schwarz' theorem holds (cf. \cite{CarmonaDelarue2018II} Remark 4.16). Therefore, we only use the $\partial_x\partial_\mu$-terms instead of both $\partial_\mu\partial_x$ and $\partial_x\partial_\mu$.
    \end{enumerate}
\end{remark}
We get the following lemma, which establishes the order of the variational processes in $\epsilon$, that will be crucial for the derivation of the maximum principle from the expansion of the cost functional.
\begin{lemma}\label{lemma:Estimates on Y and Z}
    If Assumption \ref{Assumption:C^{2,1}} holds, then, for any $k \geq 1$,
    \begin{align}
        &\mathbb{E}\left[\sup _{t \in[0, T]}\left\|\Delta X_t\right\|^{2 k}\right] \in O(\epsilon^k), \label{eq:DeltaXEstimate}\\
        &\mathbb{E}\left[\sup _{t \in[0, T]}\left\|Y_t\right\|^{2 k}\right] \in O(\epsilon^k), \label{eq:FirstVariationalEstimate}\\
        &\mathbb{E}\left[\sup _{t \in[0, T]}\left\|Z_t\right\|^{2 k}\right] \in O(\epsilon^{2k}), \label{eq:SecondVariationalEstimate} \\
        &\mathbb{E}\left[\sup _{t \in[0, T]}\left\|\Delta X_t-Y_t\right\|^{2 k}\right] \in O(\epsilon^{2k}). \label{eq:DeltaX minus first Variational Estimate}
    \end{align}
\end{lemma}
\begin{proof}
    Towards \eqref{eq:DeltaXEstimate}, using Ito's formula, we see that
    \begin{align*}
        \mathbb{E}\left[\sup _{t \in[0, T]}\left\|X_t^{\epsilon}-X_t\right\|^{2 k}\right]
        =& \E\Bigg[\bigg(\sup_{t\in[0,T]}\int_0^t 2\left\langle A(s,X^\epsilon_s,\mu^\epsilon_s,\alpha^\epsilon_s)-A(s,X_s,\mu_s,\alpha^\epsilon_s)+\delta A(s), X^\epsilon_s-X_s\right\rangle \\
        &+\left\| B(s,X^\epsilon_s,\mu^\epsilon_s,\alpha^\epsilon_s)-B(s,X_s,\mu_s,\alpha^\epsilon_s)+\delta B(s)\right\|_{L_2(\R^d)}^2\\
        &+ \left\| C(s,X^\epsilon_s,\mu^\epsilon_s,\alpha^\epsilon_s)-C(s,X_s,\mu_s,\alpha^\epsilon_s)+\delta C(s)\right\|_{L_2(\R^d)}^2 \ds\\
        &+2\sup_{t\in[0,T]} \int_0^t \langle X^\epsilon_s-X_s, B(s,X^\epsilon_s,\mu^\epsilon_s,\alpha^\epsilon_s)-B(s,X_s,\mu_s,\alpha^\epsilon_s)+\delta B(s) \dW^1_s\rangle \\
        &+2\sup_{t\in[0,T]} \int_0^t \langle X^\epsilon_s-X_s, C(t,X^\epsilon_s,\mu^\epsilon_s,\alpha^\epsilon_s)-C(t,X_s,\mu_s,\alpha^\epsilon_s)+\delta C(s) \dW^0_s\rangle\bigg)^k\Bigg].
    \end{align*}
    Now, using the Cauchy-Schwarz and BDG inequalities, 
    using the conditions on the coefficients, the Young inequality and $\W_2(\mu_t,\mu^\epsilon_t)^2\leq \tilde{\E}\left[\|\tilde{X}^\epsilon_t-\tilde{X}_t\|^2\right]$, we get
    \begin{align*}
        &\mathbb{E}\left[\sup _{t \in[0, T]}\left\|X_t^{\epsilon}-X_t\right\|^{2 k}\right]\\
        \leq & c_1\E\Bigg[\bigg(\int_0^T \left( \| X^\epsilon_s-X_s\|+\W_2(\mu_s,\mu^\epsilon_s)\right)\| X^\epsilon_s-X_s\|
        +\| \delta A(s)\|\| X^\epsilon_s-X_s\|\\
        &+\left( \| X^\epsilon_s-X_s\|+\W_2(\mu_s,\mu^\epsilon_s)\right)^2+\left\|\delta B(s)\right\|_{L_2(\R^d)}^2
        +\left\|\delta C(s)\right\|_{L_2(\R^d)}^2 \ds\bigg)^k\Bigg]\\
        &+\E\Bigg[\bigg(c_2\sup_{s\in[0,T]}\| X^\epsilon_s-X_s\|^2+c_3\int_0^T \left( \| X^\epsilon_s-X_s\|+\W_2(\mu_s,\mu^\epsilon_s)\right)^2+\left\|\delta B(s)\right\|_{L_2(\R^d)}^2 \ds \\
        &+c_2\sup_{s\in[0,T]}\| X^\epsilon_s-X_s\|^2+c_3\int_0^T \left( \| X^\epsilon_s-X_s\|+\W_2(\mu_s,\mu^\epsilon_s)\right)^2+\left\|\delta C(s)\right\|_{L_2(\R^d)}^2 \ds\bigg)^k\Bigg]\\
        \leq & c_4\E\left[\int_0^T \sup_{s\in[0,t]}\| X^\epsilon_s-X_s\|^{2k} \dt\right]+\frac{1}{2} \E\left[\sup_{s\in[0,T]}\| X^\epsilon_s-X_s\|^{2k}\right]
        \\
        &+c_5\E\left[\left(\int_{E_\epsilon}\| \delta A(s)\|^2+\left\|\delta B(s)\right\|_{L_2(\R^d)}^2
        +\left\|\delta C(s)\right\|_{L_2(\R^d)}^2 \ds\right)^k\right],
    \end{align*}
    and the last term is of order $O(\epsilon^k)$, due to the boundedness of the coefficients and since the $\delta$-terms vanish outside $E_\epsilon$. Thus, Gronwall gives \eqref{eq:DeltaXEstimate}.\\
    Towards \eqref{eq:FirstVariationalEstimate}, using Ito's formula, we get
    \begin{align*}
        \E\left[\sup_{t\in [0,T]}\|Y_t\|^{2k}\right]
        =& \E\Bigg[\bigg(\sup_{t\in[0,T]}\int_0^t 2\left\langle A_x(\theta_s)Y_s+\tilde{\E}\left[A_\mu(\theta_s)(\tilde{X}_s) \tilde{Y}_s\right]
        + \delta A(s), Y_s\right\rangle  \\
        &+ \left\| B_x(\theta_s)Y_s + \tilde{\E}\left[B_\mu(\theta_s)(\tilde{X}_s) \tilde{Y}_s\right] 
        + \delta B(s)\right\|_{L_2(\R^d)}^2\\
        &+\left\| C_x(\theta_s)Y_s + \tilde{\E}\left[C_\mu(\theta_s)(\tilde{X}_s) \tilde{Y}_s\right] 
        + \delta C(s) \right\|_{L_2(\R^d)}^2\ds\\
        &+2\sup_{t\in[0,T]} \int_0^t \langle Y_s, B_x(\theta_s)Y_s + \tilde{\E}\left[B_\mu(\theta_s)(\tilde{X}_s) \tilde{Y}_s\right] 
        + \delta B(s)  \dW^1_s\rangle\\
        &+2\sup_{t\in[0,T]} \int_0^t \langle Y_s, C_x(\theta_s)Y_s + \tilde{\E}\left[C_\mu(\theta_s)(\tilde{X}_s) \tilde{Y}_s\right] 
        + \delta C(s)  \dW^0_s\rangle\bigg)^k\Bigg].
    \end{align*}
    Again, using the BDG inequality, Young, Cauchy-Schwarz and the assumptions on the coefficients
    \begin{align*}
        \E\left[\sup_{t\in [0,T]}\|Y_t\|^{2k}\right]
        \leq & c_6\E\Bigg[\bigg(\int_0^T \| A_x(\theta_s)\|\|Y_s\|^2+\tilde{\E}\left[\|A_\mu(\theta_s)(\tilde{X}_s)\|\|\tilde{Y}_s\|\right]\|Y_s\|
        +\|\delta A(s)\|\|Y_s\|  \\
        &+ \left\| B_x(\theta_s)Y_s + \tilde{\E}\left[B_\mu(\theta_s)(\tilde{X}_s) \tilde{Y}_s\right] 
        + \delta B(s)\right\|_{L_2(\R^d)}^2\\
        &+\left\| C_x(\theta_s)Y_s + \tilde{\E}\left[C_\mu(\theta_s)(\tilde{X}_s) \tilde{Y}_s\right] 
        + \delta C(s) \right\|_{L_2(\R^d)}^2\ds\bigg)^k\Bigg]\\
        &+c_7\E\Bigg[\bigg(\int_0^T \|Y_s\|^2\| B_x(\theta_s)Y_s + \tilde{\E}\left[B_\mu(\theta_s)(\tilde{X}_s) \tilde{Y}_s\right] 
        + \delta B(s)\|_{L_2(\R^d)}^2 \ds\bigg)^{\frac{k}{2}}\Bigg]\\
        &+c_7\E\Bigg[\bigg(\int_0^T \|Y_s\|^2\| C_x(\theta_s)Y_s + \tilde{\E}\left[C_\mu(\theta_s)(\tilde{X}_s) \tilde{Y}_s\right] 
        + \delta C(s)\|_{L_2(\R^d)}^2  \ds\bigg)^{\frac{k}{2}}\Bigg]\\
        \leq & c_8\E\Bigg[\int_0^T \sup_{s\in[0,t]}\|Y_s\|^{2k} \dt+\frac{1}{2}\E\Bigg[\sup_{t\in [0,T]}\|Y_s\|^{2k}\Bigg]\\
        &+c_9\E\Bigg[\bigg(\int_{E_\epsilon} \|\delta A(s)\|^2  
        + \|\delta B(s)\|_{L_2(\R^d)}^2
        +\left\|\delta C(s) \right\|_{L_2(\R^d)}^2\ds\bigg)^k\Bigg].
    \end{align*}
    and the last term is of order $O(\epsilon^k)$, due to the boundedness of the coefficients and since the $\delta$-terms vanish outside $E_\epsilon$. Thus, Gronwall gives \eqref{eq:FirstVariationalEstimate}.\\
    Similarly, towards \eqref{eq:SecondVariationalEstimate}, using Ito's formula, 
    BDG and Cauchy-Schwarz, 
    Young and the assumptions (boundedness) of the coefficients and then Fubini, we arrive at
    \begin{align*}
        &\E\left[\sup_{t\in [0,T]}\|Z_t\|^{2k}\right]\\
        \leq & c\E\left[\int_0^T\sup_{s\in[0,t]}\|Z_s\|^{2k} \dt\right]+\frac{1}{2}\E\left[\sup_{t\in[0,T]} \|Z_t\|^{2k}\right]\\
        &+\E\left[\int_0^T\|Y_s\|^{4k}\ds\right]\\
        &+\E\left[\sup_{s\in[0,T]}\|Y_s\|^{2k}\left(\int_{E_\epsilon} \|\delta A_x(s) \|^2\ds\right)^k\right]
        +\E\left[\tilde{\E}\left[\sup_{s\in[0,T]}\|\tilde{Y}_s\|^{2k}\left(\int_{E_\epsilon} \|\delta A_\mu(s)(\tilde{X}_s) \|^2\ds\right)^k\right]\right]\\
        &+\E\left[\sup_{s\in[0,T]}\|Y_s\|^{2k}\left(\int_{E_\epsilon} \|\delta B_x(s) \|^2\ds\right)^k\right]
        +\E\left[\tilde{\E}\left[\sup_{s\in[0,T]}\|\tilde{Y}_s\|^{2k}\left(\int_{E_\epsilon} \|\delta B_\mu(s)(\tilde{X}_s) \|^2\ds\right)^k\right]\right]\\
        &+\E\left[\sup_{s\in[0,T]}\|Y_s\|^{2k}\left(\int_{E_\epsilon} \|\delta C_x(s) \|^2\ds\right)^k\right]
        +\E\left[\tilde{\E}\left[\sup_{s\in[0,T]}\|\tilde{Y}_s\|^{2k}\left(\int_{E_\epsilon} \|\delta C_\mu(s)(\tilde{X}_s) \|^2\ds\right)^k\right]\right],
    \end{align*}
    so Gronwall gives \eqref{eq:SecondVariationalEstimate}, as the last seven expectations are of order $O(\epsilon^{2k)}$.\\
    Finally, towards \eqref{eq:DeltaX minus first Variational Estimate}, we denote $K=X^\epsilon-X-Y$, $X^{\lambda, \epsilon}:=X+\lambda\left(X^\epsilon-X\right)$ and $\theta^{\lambda,\epsilon}=(t,X^{\lambda,\epsilon},\cL(X^{\lambda,\epsilon}\mid \cF^0_t),\alpha^\epsilon)$. 
    We get
    \begin{align*}
        dK_t=
        & A(t,X^\epsilon_t,\mu^\epsilon_t,\alpha^\epsilon_t)-A(t,X_t,\mu_t,\alpha^\epsilon_t)-A_x(\theta_t)Y_t -\tilde{\E}\left[A_\mu(\theta_t)(\tilde{X}_t) \tilde{Y}_t\right]  \dt\\
        &+B(t,X^\epsilon_t,\mu^\epsilon_t,\alpha^\epsilon_t)-B(t,X_t,\mu_t,\alpha^\epsilon_t)-B_x(\theta_t)Y_t -\tilde{\E}\left[B_\mu(\theta_t)(\tilde{X}_t) \tilde{Y}_t\right]  \dW^1_t\\
        &+C(t,X^\epsilon_t,\mu^\epsilon_t,\alpha^\epsilon_t)-C(t,X_t,\mu_t,\alpha^\epsilon_t)-C_x(\theta_t)Y_t -\tilde{\E}\left[C_\mu(\theta_t)(\tilde{X}_t) \tilde{Y}_t\right]  \dW^0_t\\
        =
        & \int_0^1 A_x(\theta^{\lambda,\epsilon}_t)(X^\epsilon_t-X_t)+\tilde{\E}\left[A_\mu(\theta^{\lambda,\epsilon}_t)(\tilde{X}^{\lambda,\epsilon}_t)(\tilde{X}^\epsilon_t-\tilde{X}_t)\right] \dlambda\\
        &-A_x(\theta_t)(X^\epsilon_t-X_t) -\tilde{\E}\left[A_\mu(\theta_t)(\tilde{X}_t) (\tilde{X}^\epsilon_t-\tilde{X}_t)\right]\\
        &+A_x(\theta_t)K_t +\tilde{\E}\left[A_\mu(\theta_t)(\tilde{X}_t) \tilde{K}_t\right]  dt\\
        &+\int_0^1 B_x(\theta^{\lambda,\epsilon}_t)(X^\epsilon_t-X_t)+\tilde{\E}\left[B_\mu(\theta^{\lambda,\epsilon}_t)(\tilde{X}^{\lambda,\epsilon}_t)(\tilde{X}^\epsilon_t-\tilde{X}_t)\right] \dlambda\\
        &-B_x(\theta_t)(X^\epsilon_t-X_t) -\tilde{\E}\left[B_\mu(\theta_t)(\tilde{X}_t) (\tilde{X}^\epsilon_t-\tilde{X}_t)\right]\\
        &+B_x(\theta_t)K_t +\tilde{\E}\left[B_\mu(\theta_t)(\tilde{X}_t) \tilde{K}_t\right]  \dW^1_t\\
        &+\int_0^1 C_x(\theta^{\lambda,\epsilon}_t)(X^\epsilon_t-X_t)+\tilde{\E}\left[C_\mu(\theta^{\lambda,\epsilon}_t)(\tilde{X}^{\lambda,\epsilon}_t)(\tilde{X}^\epsilon_t-\tilde{X}_t)\right] \dlambda\\
        &-C_x(\theta_t)(X^\epsilon_t-X_t) -\tilde{\E}\left[C_\mu(\theta_t)(\tilde{X}_t) (\tilde{X}^\epsilon_t-\tilde{X}_t)\right]\\
        &+C_x(\theta_t)K_t +\tilde{\E}\left[C_\mu(\theta_t)(\tilde{X}_t) \tilde{K}_t\right]  \dW^0_t
    \end{align*}
    Using Ito's formula gives
    \begin{align*}
        &\E\left[\sup_{t\in [0,T]}\|K_s\|^{2k}\right]\\
        &\leq  \E\Bigg[\bigg(\sup_{t\in[0,T]}\int_0^t 2\left\langle A_x(\theta_s)K_s +\tilde{\E}\left[A_\mu(\theta_s)(\tilde{X}_s) \tilde{K}_s\right] , K_s\right\rangle \\
        &+ \left\| B_x(\theta_s)K_s +\tilde{\E}\left[B_\mu(\theta_s)(\tilde{X}_s) \tilde{K}_s\right] \right\|_{L_2(\R^d)}^2
        + \left\| C_x(\theta_s)K_s +\tilde{\E}\left[C_\mu(\theta_s)(\tilde{X}_s) \tilde{K}_s\right] \right\|_{L_2(\R^d)}^2 \ds\\
        &+2\sup_{t\in[0,T]} \int_0^t \langle K_s, B_x(\theta_s)K_s +\tilde{\E}\left[B_\mu(\theta_s)(\tilde{X}_s) \tilde{K}_s\right]  dW^1_s\rangle\\
        &+2\sup_{t\in[0,T]} \int_0^t \langle K_s, C_x(\theta_s)K_s +\tilde{\E}\left[C_\mu(\theta_s)(\tilde{X}_s) \tilde{K}_s\right]  dW 0_s\rangle\\
        &+\sup_{t\in[0,T]}\int_0^t 2\langle \int_0^1 \left(A_x(\theta^{\lambda,\epsilon}_s)-A_x(\theta_s)\right)(X^\epsilon_s-X_s)\\
        &+\tilde{\E}\left[\left(A_\mu(\theta^{\lambda,\epsilon}_s)(\tilde{X}^{\lambda,\epsilon}_s)-A_\mu(\theta_s)(\tilde{X}_s)\right)(\tilde{X}^\epsilon_s-\tilde{X}_s)\right] \dlambda, K_s\rangle\\
        &+ \Big\|\int_0^1 \left(B_x(\theta^{\lambda,\epsilon}_s)-B_x(\theta_s)\right)(X^\epsilon_s-X_s)
        +\tilde{\E}\left[\left(B_\mu(\theta^{\lambda,\epsilon}_s)(\tilde{X}^{\lambda,\epsilon}_s)-B_\mu(\theta_s)(\tilde{X}_s)\right)(\tilde{X}^\epsilon_s-\tilde{X}_s)\right] \dlambda \Big\|^2_{L_2(\R^d)}\\
        &+ \Big\|\int_0^1 \left(C_x(\theta^{\lambda,\epsilon}_s)-C_x(\theta_s)\right)(X^\epsilon_s-X_s)
        +\tilde{\E}\left[\left(C_\mu(\theta^{\lambda,\epsilon}_s)(\tilde{X}^{\lambda,\epsilon}_s)-C_\mu(\theta_s)(\tilde{X}_s)\right)(\tilde{X}^\epsilon_s-\tilde{X}_s)\right] \dlambda \Big\|^2_{L_2(\R^d)} \ds\\
        &+2\sup_{t\in[0,T]}\Big( \int_0^t \langle K_s, \int_0^1 (B_x(\theta^{\lambda,\epsilon}_s)-B_x(\theta_s))(X^\epsilon_s-X_s)
        +\tilde{\E}[(B_\mu(\theta^{\lambda,\epsilon}_s)(\tilde{X}^{\lambda,\epsilon}_s)-B_\mu(\theta_s)(\tilde{X}_s))(\tilde{X}^\epsilon_s-\tilde{X}_s)] \dlambda  dW^1_s\rangle \\
        &+ \int_0^t \langle K_s, \int_0^1 (C_x(\theta^{\lambda,\epsilon}_s)-C_x(\theta_s))(X^\epsilon_s-X_s)
        +\tilde{\E}[(C_\mu(\theta^{\lambda,\epsilon}_s)(\tilde{X}^{\lambda,\epsilon}_s)-C_\mu(\theta_s)(\tilde{X}_s))(\tilde{X}^\epsilon_s-\tilde{X}_s)] \dlambda  dW^0_s\rangle \Big)\bigg)^k \Bigg],
    \end{align*}
    where the first four lines can be estimated by Cauchy-Schwarz, Young and the boundedness assumptions on the coefficients, and the rest can be estimated using BDG/Cauchy-Schwarz, then Young to separate $K$ and then the Lipschitz-continuity of the coefficients and the already established order of $X^\epsilon-X$. Then Gronwall, gives \eqref{eq:DeltaX minus first Variational Estimate}.
\end{proof}
The most important estimate is the following.
\begin{lemma}\label{lemma:estimate on DeltaX-Y-Z}
    If Assumption \ref{Assumption:C^{2,1}} holds, then for any $1 \leq k$,
    \begin{equation*}
        \mathbb{E}\left[\sup _{t \in[0, T]}\left\|X^{\epsilon}_t-X_t-Y^{\epsilon}_t-Z^{\epsilon}_t\right\|^{2 k}\right] \in o(\epsilon^{2k}).
    \end{equation*}
\end{lemma}
\begin{proof}
    We write $K_t:=X^{\epsilon}_t-X_t-Y_t-Z_t$. 
    Clearly, by our previous estimates from Lemma \ref{lemma:Estimates on Y and Z}
    \begin{equation}
        \E\left[\sup_{t\in[0,T]}\|K_t\|^{2k}\right]\in O(\epsilon^{2k}).\label{eq:Apriori Estimate on K}
    \end{equation}
    Defining
    \begin{equation}\begin{aligned}
        K^{(1)}_t
        =& A(t,X^\epsilon_t,\mu^\epsilon_t,\alpha^\epsilon_t)
        -A(t,X_t,\mu_t,\alpha^\epsilon_t)
        - A_x(\theta_t)[Y_t+Z_t] 
        -\tilde{\E}\left[A_\mu(\theta_t)(\tilde{X}_t) [\tilde{Y}_t+\tilde{Z}_t]\right] 
        \\
        &- \frac{1}{2} A_{xx}(\theta_t)[Y_t,Y_t] 
        - \frac{1}{2}\tilde{\E}\left[ A_{y\mu}(\theta_t)(\tilde{X}_t)[\tilde{Y}_t,\tilde{Y}_t]\right]  \\
        & - \frac{1}{2}\tilde{\E} \left[\tilde{\tilde\E}\left[ A_{\mu\mu}(\theta_t)(\tilde{X}_t,\tilde{\tilde{X}}_t)[\tilde{Y}_t,\tilde{\tilde{Y}}_t]\right]\right] 
        -\tilde{\E}\left[ A_{x\mu}(\theta_t)(\tilde{X}_t)[\tilde{Y}_t,Y_t]\right]\\
        &- \delta A_x(t) Y_t
        -\tilde{\E}\left[\delta A_\mu(t)(\tilde{X}_t) \tilde{Y}_t\right]  
    \end{aligned}\label{eq:Drift of K}\end{equation}
    and $K^{(2)}, K^{(3)}$ the corresponding terms for $A$ replaced by $B$ and $C$ respectively, we can write
    \begin{equation*}
        \mathrm{d}K_t= K^{(1)}_t \dt + K^{(2)}_t \dW^1_t+ K^{(3)}_t \dW^0_t.
    \end{equation*}
    Again, using the notation $X^{\lambda, \epsilon}:=X+\lambda\left(X^\epsilon-X\right)$ and $\theta^{\lambda,\epsilon}=(t,X^{\lambda,\epsilon},\cL(X^{\lambda,\epsilon}\mid \cF^0_t),\alpha^\epsilon)$, we notice that
    \begin{equation}\begin{aligned}
        &A(t,X^\epsilon_t,\mu^\epsilon_t,\alpha^\epsilon_t)
        -A(t,X_t,\mu_t,\alpha^\epsilon_t)
        - A_x(\theta_t)[Y_t+Z_t] 
        -\tilde{\E}\left[A_\mu(\theta_t)(\tilde{X}_t) [\tilde{Y}_t+\tilde{Z}_t]\right]\\
        =&\int_0^1 A_x(\theta^{\lambda,\epsilon}_t)[K_t]
        +\tilde{\E}[A_\mu(\theta^{\lambda,\epsilon}_t)(\tilde{X}^{\lambda,\epsilon}_t)[\tilde{K}_t]]\\
        &+(A_x(\theta^{\lambda,\epsilon}_t)-A_x(\theta_t))[Y_t+Z_t]
        +\tilde{\E}[(A_\mu(\theta^{\lambda,\epsilon}_t)(\tilde{X}^{\lambda,\epsilon}_t)-A_\mu(\theta_t)(\tilde{X}_t))[\tilde{Y}_t+\tilde{Z}_t]]\dlambda.
    \end{aligned}\label{eq:first expansion of drift of K}\end{equation}
    Expanding further, we get
    \begin{equation}\begin{aligned}
        A_x(\theta^{\lambda,\epsilon}_t)-A_x(\theta_t)
        =& \lambda \int_0^1 A_{xx}(\theta^{\lambda\gamma,\epsilon}_t)[K_t]
        +\tilde{\E}\left[A_{x\mu}(\theta^{\lambda\gamma,\epsilon}_t)(\tilde{X}^{\lambda\gamma,\epsilon}_t)[\tilde{K}_t]\right]\dgamma +\delta A_x(t)\\
        &+\lambda \int_0^1 A_{xx}(\theta^{\lambda\gamma,\epsilon}_t)[Y_t+Z_t]
        +\tilde{\E}\left[A_{x\mu}(\theta^{\lambda\gamma,\epsilon}_t)(\tilde{X}^{\lambda\gamma,\epsilon}_t)[\tilde{Y}_t+\tilde{Z}_t]\right]\dgamma,
    \end{aligned}\label{eq:second expansion of drift of K in x}\end{equation}
    and
    \begin{equation}\begin{aligned}
        &A_\mu(\theta^{\lambda,\epsilon}_t)(\tilde{X}^{\lambda,\epsilon}_t)-A_\mu(\theta_t)(\tilde{X}_t)\\
        =& \lambda \int_0^1 \tilde{\tilde{\E}}\left[A_{\mu\mu}(\theta^{\lambda\gamma,\epsilon}_t)(\tilde{X}^{\lambda\gamma,\epsilon}_t,\tilde{\tilde{X}}^{\lambda\gamma,\epsilon}_t)[\tilde{\tilde{K}}_t]\right]
        +A_{x\mu}(\theta^{\lambda\gamma,\epsilon}_t)(\tilde{X}^{\lambda\gamma,\epsilon}_t)[K_t]
        +A_{y\mu}(\theta^{\lambda\gamma,\epsilon}_t)(\tilde{X}^{\lambda\gamma,\epsilon}_t)[\tilde{K}_t] \\
        &+\tilde{\tilde{\E}}\left[A_{\mu\mu}(\theta^{\lambda\gamma,\epsilon}_t)(\tilde{X}^{\lambda\gamma,\epsilon}_t,\tilde{\tilde{X}}^{\lambda\gamma,\epsilon}_t)[\tilde{\tilde{Y}}_t+\tilde{\tilde{Z}}_t]\right]
        +A_{x\mu}(\theta^{\lambda\gamma,\epsilon}_t)(\tilde{X}^{\lambda\gamma,\epsilon}_t)[Y_t+Z_t]\\
        &+A_{y\mu}(\theta^{\lambda\gamma,\epsilon}_t)(\tilde{X}^{\lambda\gamma,\epsilon}_t)[\tilde{Y}_t+\tilde{Z}_t]\dgamma 
        +\delta A_\mu(t)(\tilde{X}_t).
    \end{aligned}\label{eq:second expansion of drift of K in mu}\end{equation}
    Plugging \eqref{eq:second expansion of drift of K in x} and \eqref{eq:second expansion of drift of K in mu} into \eqref{eq:first expansion of drift of K} and the resulting terms into \eqref{eq:Drift of K} and using Schwarz' theorem (cf. \cite{CarmonaDelarue2018II} Remark 4.16), results in
    \begin{align}
        &K^{(1)}_t\notag\\
        =&\int_0^1 A_x(\theta^{\lambda,\epsilon}_t)[K_t]
        +\tilde{\E}[A_\mu(\theta^{\lambda,\epsilon}_t)(\tilde{X}^{\lambda,\epsilon}_t)[\tilde{K}_t]]\dlambda \label{eq:K:Gronwall terms}\\
        &+ \delta A_x(t) Z_t
        +\tilde{\E}\left[\delta A_\mu(t)(\tilde{X}_t) \tilde{Z}_t\right]\label{eq:K:delta terms}\\
        &+\int_0^1\lambda\int_0^1 A_{xx}(\theta^{\lambda\gamma,\epsilon}_t)[K_t,Y_t+Z_t]
        +\tilde{\E}\bigg[\tilde{\tilde{\E}}\left[A_{\mu\mu}(\theta^{\lambda\gamma,\epsilon}_t)(\tilde{X}^{\lambda\gamma,\epsilon}_t,\tilde{\tilde{X}}^{\lambda\gamma,\epsilon}_t)[\tilde{\tilde{K}}_t,\tilde{Y}_t+\tilde{Z}_t]\right]\label{eq:K:Higher order terms 1}\\
        &+2A_{x\mu}(\theta^{\lambda\gamma,\epsilon}_t)(\tilde{X}^{\lambda\gamma,\epsilon}_t)[\tilde{K}_t,Y_t+Z_t]
        +A_{y\mu}(\theta^{\lambda\gamma,\epsilon}_t)(\tilde{X}^{\lambda\gamma,\epsilon}_t)[\tilde{K}_t,\tilde{Y}_t+\tilde{Z}_t]\bigg]\dgamma\dlambda \label{eq:K:Higher order terms 2}\\
        &+\int_0^1\lambda\int_0^1  A_{xx}(\theta^{\lambda\gamma,\epsilon}_t)[Y_t+Z_t,Y_t+Z_t]
        -A_{xx}(\theta_t)[Y_t,Y_t]\label{eq:K:Y^2 terms 1}\\
        &+2\tilde{\E}\left[A_{x\mu}(\theta^{\lambda\gamma,\epsilon}_t)(\tilde{X}^{\lambda\gamma,\epsilon}_t)[\tilde{Y}_t+\tilde{Z}_t,Y_t+Z_t]
        -A_{x\mu}(\theta_t)(\tilde{X}_t)[\tilde{Y}_t,Y_t]\right]\label{eq:K:Y^2 terms 2}\\
        &+\tilde{\E}\Big[\tilde{\tilde{\E}}\big[
        A_{\mu\mu}(\theta^{\lambda\gamma,\epsilon}_t)(\tilde{X}^{\lambda\gamma,\epsilon}_t,\tilde{\tilde{X}}^{\lambda\gamma,\epsilon}_t)[\tilde{\tilde{Y}}_t+\tilde{\tilde{Z}}_t,\tilde{Y}_t+\tilde{Z}_t]
        -A_{\mu\mu}(\theta_t)(\tilde{X}_t,\tilde{\tilde{X}}_t)[\tilde{Y}_t,\tilde{\tilde{Y}}_t]\big]\Big]\label{eq:K:Y^2 terms 3}\\
        &+\tilde{\E}\big[ A_{y\mu}(\theta^{\lambda\gamma,\epsilon}_t)(\tilde{X}^{\lambda\gamma,\epsilon}_t)[\tilde{Y}_t+\tilde{Z}_t,\tilde{Y}_t+\tilde{Z}_t]
        -A_{y\mu}(\theta_t)(\tilde{X}_t)[\tilde{Y}_t,\tilde{Y}_t]\big]\dgamma\dlambda \label{eq:K:Y^2 terms 4}.
    \end{align}
    For $K^{(2)}$ and $K^{(3)}$ the corresponding formulas with $A$ replaced by $B$ and $C$ respectively hold. Now, $K$ can be estimated.
    Itô's formula gives
    \begin{align*}
        \mathbb{E}\left[\sup _{t \in[0, T]}\left\|K_t\right\|^{2 k}\right]
        = &\E\Bigg[\bigg(\sup _{t \in[0, T]}\int_0^t 2\langle K_s,K^{(1)}_s\rangle + \|K^{(2)}_s\|^2+\|K^{(3)}_s\|^2 \ds\\
        &+\sup _{t \in[0, T]}\int_0^t \langle K_s,K^{(2)}_s \dW^1\rangle 
        +\sup _{t \in[0, T]}\int_0^t \langle K_s,K^{(3)}_s \dW^0\rangle \bigg)^k\Bigg]
    \end{align*}
    After using the Burkholder-Davis-Gundy and standard arguments, one can treat the terms in the following way. As terms containing $K^{(1)}$ are treated in the same way as the respective terms for $K^{(2)}$ and $K^{(3)}$, we only argue for $K^{(1)}$. The terms containing \eqref{eq:K:Gronwall terms} are treated by a Gronwall argument. The terms containing \eqref{eq:K:delta terms} are of high enough order by Lemma \ref{lemma:Estimates on Y and Z}. The terms containing \eqref{eq:K:Higher order terms 1} and \eqref{eq:K:Higher order terms 2} are of high enough order by \eqref{eq:Apriori Estimate on K} and Lemma \ref{lemma:Estimates on Y and Z}. For \eqref{eq:K:Y^2 terms 1}, the terms containing $Z$ are again of higher order by Lemma \ref{lemma:Estimates on Y and Z}, while for the terms containing $[Y,Y]$ the difference gives the right order by the Lipschitz assumption on $A$ and $B$ from Assumption \ref{Assumption:C^{2,1}} and the estimate \eqref{eq:DeltaXEstimate}. Terms containing \eqref{eq:K:Y^2 terms 2}, \eqref{eq:K:Y^2 terms 3} and \eqref{eq:K:Y^2 terms 4} are argued in the same way.
\end{proof}

\section{The Adjoint Equations}\label{sec:The Adjoint Equations}
The first-order adjoint process $(p,q,r)$ is defined as a solution to
\begin{equation}
    \begin{aligned}
        \mathrm{d}p_t 
        =& - \Big(A^*_x(\theta_t) p_t + B^*_x(\theta_t) q_t+C^*_x(\theta_t) r_t +f^*_x(\theta_t)\\
        &+ \tilde{\E}\left[ A^*_\mu(\tilde{\theta}_t)(X_t) \tilde{p}_t 
        + B^*_\mu(\tilde{\theta}_t)(X_t) \tilde{q}_t
        + C^*_\mu(\tilde{\theta}_t)(X_t) \tilde{r}_t
        +f^*_\mu(\tilde{\theta}_t)(X_t)\right]\Big) \dt  \\
        &\quad + q_t \dW^1_t+ r_t \dW^0_t, \\
        p_T &= g_x(X_T,\mu_T) + \tilde{\E}[g_\mu(\tilde{X}_T,\mu_T)(X_T)]. 
    \end{aligned}\label{eq:First order adjoint}
\end{equation}
Now, let $(e_n)_{n=1}^d$ be a basis of $\R^d$. Note that $B_x(\theta_t):\R^d\to \R^{d\times d}$ linearly, i.e. $B_x(\theta_t):\R^d\times \R^d\to \R^d, (y,w) \mapsto B_x(\theta_t)[y,w]$ bilinearly, implying $B_x(\theta_t)[\cdot,w]\in \R^{d\times d}$. Denote $B_x^k(t):=B_x(\theta_t)[\cdot,e_k]$ and $Q^k_t:= Q_t e_k \in \R^{d\times d}$. The first-level second-order adjoint process $(P,Q,R)$ is defined as a solution to
\begin{equation}
    \begin{aligned}
        \mathrm{d}P_t =& -\Big(A^*_x(\theta_t) P_t +P_t A_x(\theta_t)
        +\sum_{k}B^{k,*}_x(\theta_t) P_tB^k_x(\theta_t) + \sum_{k}B^{k,*}_x(\theta_t) Q^k_t + \sum_{k}Q^k_t B^k_x(\theta_t)\\
        &+\sum_{k}C^{k,*}_x(\theta_t) P_tC^k_x(\theta_t) + \sum_{k}C^{k,*}_x(\theta_t) R^k_t + \sum_{k}R^k_t C^k_x(\theta_t)\\
        &
        +H_{xx}(\theta_t)
        + \tilde{\E}[ H_{y\mu}(\tilde{\theta}_t)(X_t) ] \Big)\dt
        + Q_t \dW^1_t
        + R_t \dW^0_t, \\
        P_T =& g_{xx}(X_T,\mu_T) + \tilde{\E}[g_{y\mu}(\tilde{X}_T,\mu_T)( X_T)]
    \end{aligned}
    \label{eq:First-Level second-order adjoint equation}
\end{equation}
where 
\begin{equation}
    H(t,x, \mu, \alpha, p, q, r):=\langle A(t,x, \mu, \alpha), p\rangle + \langle B(t,x, \mu, \alpha) ,q\rangle +\langle C(t,x, \mu, \alpha),r\rangle +f(t,x, \mu, \alpha),
\end{equation}
and we expand our $\theta$-notation to $\theta_t=(t,X_t,\mu_t,\alpha_t,p_t,q_t,r_t)$ and remember the convention from the end of Section \ref{sec:Preliminaries} that only needed arguments are taken.
\begin{remark}\label{remark:after first two adjoints, existence, symmetry}
    \begin{enumerate}[(i)]
        \item \eqref{eq:First order adjoint} is a conditional McKean--Vlasov BSDE and existence and uniqueness of a solution are discussed in the Appendix \ref{appendix:Ex and Uni for Controlled MKVSDE with Common Noise} and follow by Theorem \ref{theorem_appendix:existence and uniqueness of Cond-MKVBSDE}.
        \item \eqref{eq:First-Level second-order adjoint equation} is a standard BSDE and has a unique adapted solution by already known results (cf. \cite{Pardoux2014} chapter 5).
        \item The law dependence in \eqref{eq:First order adjoint} and \eqref{eq:First-Level second-order adjoint equation} is again on the joint conditional laws $\cL(X_t,p_t\mid\cF^0_t)$/$\cL(X_t,P_t\mid\cF^0_t)$ and not on $\cL(p_t\mid\cF^0_t)$/$\cL(P_t\mid\cF^0_t)$ alone, respectively.
        \item Again, all coefficients are sufficiently measurable (cf. Remark \ref{remark:after def of variational eq}).
        \item $P$, $Q^k$ and $R^k$ take values in the symmetric matrices. For this notice that $H_{xx}$ and $g_{xx}$ are symmetric as Hessians and $H_{y\mu}$ and $g_{y\mu}$ are actually also symmetric (cf. \cite{CarmonaDelarue2018I} Corollary 5.89).
    \end{enumerate}
\end{remark}
\subsection{The Lifting and the Second-Level Second-Order Adjoint Equation} \label{subsec:Lifting and Third Adjoint}
    Now, we come to the third adjoint equation and its setting. This approach was first introduced by \citet{spille2026}. We will adapt the setting to fit the common noise case.\\
    So far, we have worked on $(\Omega,\cF,(\cF_t)_{t\geq 0 },\P):=\bigotimes_{i=0}^1 (\Omega^i,\mathcal{F}^i,(\cF^i_t)_{t\geq 0 },\mathbb{P}^i)$ and used a copy of $(\Omega^1,\mathcal{F}^1,(\cF^1_t)_{t\geq 0 },\mathbb{P}^1)$, called $(\tilde{\Omega},\tilde{\mathcal{F}},(\tilde{\cF}_t)_{t\geq 0 },\tilde{\mathbb{P}})$, to generate copies of random variables, such that the second variable could be integrated out, thereby introducing conditional distribution dependence (cf. beginning of Section \ref{sec:The Variational Equations}). Now, we want to linearize the interaction of two copies of $Y$, which are independent in the second variable. 
    To be able to do this, we need two independent variables that are not immediately integrated out. 
    Thus, we really need to expand our probability space.\\
    We denote $(\hat{\Omega}, \hat{\cF}, (\hat{\cF}_t)_{t\geq 0 }, \hat{\P}):=(\Omega^1, \cF^1, (\cF^1_t)_{t\geq 0 }, \P^1)$ a copy and consider the product space
    \begin{equation*}
        (\bar{\Omega}, \bar{\P})
        =(\Omega,  \P)\otimes  (\hat{\Omega},  \hat{\P}),
    \end{equation*}
    equipped with the completed versions $\bar{\cF}, (\bar{\cF_t})_{t\geq 0 }$ of the product $\sigma$-algebras $(\cF, (\cF_t)_{t\geq 0 })\otimes (\hat{\cF}, (\hat{\cF}_t)_{t\geq 0 })$. We will now canonically lift all our previously defined processes to the product space by defining (with an abuse of notation) for a random variable $\Psi:\Omega\to \R^m$
    \begin{equation*}
        \Psi(\bar{\omega})=\Psi(\bar{\omega}_0,\bar{\omega}_1,\bar{\omega}_2):=\Psi(\bar{\omega}_0,\bar{\omega}_1),\quad \text{and}\quad \hat{\Psi}(\bar{\omega})=\hat{\Psi}(\bar{\omega}_0,\bar{\omega}_1,\bar{\omega}_2):=\Psi(\bar{\omega}_0,\bar{\omega}_2)
    \end{equation*}
    and so forth. The resulting expectation $\bar{\E}$ then has the same effect on the lifting as the original has on the original processes, i.e. $\bar{\E}[\Psi_t]=\E[\Psi_t]$. This lifting allows us to easily obtain copies of our process, which are independent of the previous ones in the second component conditionally on the first,
    \begin{equation*}
        \hat{W}^1(\bar{\omega}):=W^1(\bar{\omega}_0,\bar{\omega}_2),\quad
        \hat{X}(\bar{\omega}):=X(\bar{\omega}_0,\bar{\omega}_2),\quad\dots
    \end{equation*}
    For this reason we will denote $\E[\hat{\E}[\varphi(X_t,\hat{X_t})]]:=\bar{\E}[\varphi(X_t,\hat{X}_t)]$.\\
    Consequently, we will also need to use a copy of $(\Omega^1,\mathbb{P}^1)\otimes (\Omega^1,\mathbb{P}^1)$ to be able to lift processes defined on $\Bar{\Omega}$. This will still be called $(\tilde{\Omega},\tilde{\mathbb{P}})$ and will be equipped with the completed versions of the product-$\sigma$-algebras, called $(\tilde{\mathcal{F}},(\tilde{\cF}_t)_{t\geq 0 })$. For $\Phi:\bar{\Omega}\to \R^m$, we define
    \begin{align*}
        &\tilde{\Phi}:\tilde{\Omega}\times \bar{\Omega}\to \R^{m},\quad (\tilde{\omega},\bar{\omega})=(\tilde{\omega}_1,\tilde{\omega}_2,\bar{\omega}_0,\bar{\omega}_1,\bar{\omega}_2)\mapsto 
        \Phi(\bar{\omega}_0,\tilde{\omega}_1,\bar{\omega}_2)
        \quad\text{and}\\
        &\tilde{\hat{\Phi}}:\tilde{\Omega}\times \bar{\Omega}\to \R^{m},\quad (\tilde{\omega},\bar{\omega})=(\tilde{\omega}_1,\tilde{\omega}_2,\bar{\omega}_0,\bar{\omega}_1,\bar{\omega}_2)\mapsto 
        \Phi(\bar{\omega}_0,\bar{\omega}_1,\tilde{\omega}_2)\quad \text{and}\\
        &\tilde{\bar{\Phi}}:\tilde{\Omega}\times \bar{\Omega}\to \R^{m},\quad (\tilde{\omega},\bar{\omega})=(\tilde{\omega}_1,\tilde{\omega}_2,\bar{\omega}_0,\bar{\omega}_1,\bar{\omega}_2)\mapsto 
        \Phi(\bar{\omega}_0,\tilde{\omega}_1,\tilde{\omega}_2),
    \end{align*}
    i.e. one or both variables of the product space might be integrated out.
    For random variables $\Psi:\Omega\to\R^m$ that are lifted to $\bar{\Omega}$, we use the abuse of notation that
    \begin{align*}
        &\tilde{\Psi}:\tilde{\Omega}\times \bar{\Omega}\to \R^{m},\quad (\tilde{\omega},\bar{\omega})=(\tilde{\omega}_1,\tilde{\omega}_2,\bar{\omega}_0,\bar{\omega}_1,\bar{\omega}_2)\mapsto 
        \Psi(\bar{\omega}_0,\tilde{\omega}_1)
        \quad\text{and}\\
        &\tilde{\hat{\Psi}}:\tilde{\Omega}\times \bar{\Omega}\to \R^{m},\quad (\tilde{\omega},\bar{\omega})=(\tilde{\omega}_1,\tilde{\omega}_2,\bar{\omega}_0,\bar{\omega}_1,\bar{\omega}_2)\mapsto 
        \Psi(\bar{\omega}_0,\tilde{\omega}_2).
    \end{align*}
    We will use that $\bar{\E}[\varphi(X_t,\hat{X}_t)]=\E[\tilde{\E}[\varphi(X_t,\tilde{X}_t)]]=\E[\tilde{\E}[\varphi(X_t,\tilde{\hat{X}}_t)]]$.\\
    We can now define the second-level second-order adjoint process $(\mathfrak{P},\mathfrak{Q}^{(1)},\mathfrak{Q}^{(2)},\mathfrak{R})$ as a solution to 
    \begin{equation}
        \begin{aligned}
            \mathrm{d}\mathfrak{P}_t
            =&-\bigg( 
            A^*_x(\theta_t)\mathfrak{P}_t
            +\tilde{\E}\left[A^*_\mu(\tilde{\theta}_t)(X_t)\tilde{\mathfrak{P}}_t\right]
            +\mathfrak{P}_t A_x(\hat{\theta}_t)
            +\tilde{\E}\left[\tilde{\hat{\mathfrak{P}}}_t A_\mu(\tilde{\theta}_t)(\hat{X}_t)\right]\\
            &+\tilde{\E} \left[H_{\mu\mu}(\tilde{\theta}_t)(X_t,\hat{X}_t)\right]
            + H_{x\mu}(\hat{\theta}_t)(X_t)
            + H^*_{x\mu}(\theta_t)(\hat{X}_t)
            + P _t A_\mu(\theta_t)(\hat{X}_t)
            + A^*_\mu(\hat{\theta}_t)(X_t)\hat{P}_t\\
            & + B^*_x(\theta_t) P_t B_\mu(\theta_t)(\hat{X}_t) 
            + \tilde{\E}[B^*_\mu(\tilde{\theta}_t)(X_t)\tilde{P}_t B_\mu(\tilde{\theta}_t)(\hat{X}_t)]
            + B^*_\mu(\hat{\theta}_t)(X_t)\hat{P}_t B_x(\hat{\theta}_t)\\
            & + C^*_x(\theta_t) P_t C_\mu(\theta_t)(\hat{X}_t) 
            + \tilde{\E}[C^*_\mu(\tilde{\theta}_t)(X_t)\tilde{P}_t C_\mu(\tilde{\theta}_t)(\hat{X}_t)]
            + C^*_\mu(\hat{\theta}_t)(X_t)\hat{P}_t C_x(\hat{\theta}_t) \\
            &+ \sum_{k=1}^d  \Big(C^ {k,*}_x(\theta_t)\mathfrak{P}_tC^k_x(\hat{\theta}_t)
            + \tilde{\E}[C^{k,*}_x(\theta_t)\tilde{\hat{\mathfrak{P}}}_t C^k_\mu(\tilde{\hat{\theta}}_t)( \hat{X}_t)
            +C^{k,*}_\mu(\tilde{\theta}_t)(X_t)\tilde{\mathfrak{P}}_t C^k_x(\hat{\theta}_t)]\\
            &+\tilde{\bar{\E}}[C^{k,*}_\mu(\tilde{\theta}_t)(X_t)\tilde{\bar{\mathfrak{P}}}_t C^k_\mu(\tilde{\hat{\theta}}_t)( \hat{X}_t)]\Big)\\
            &+ \sum_{k=1}^d \Big(
            Q^{k}_t B^k_\mu(\theta_t)(\hat{X}_t)
            +B^{k,*}_\mu(\hat{\theta}_t)(X_t)\hat{Q}^{k}_t
            +B^{k,*}_x(\theta_t)\mathfrak{Q}^{(1),k}_t\\
            &+\tilde{\E}[B^{k,*}_\mu(\tilde{\theta}_t)(X_t)\tilde{\mathfrak{Q}}^{(1),k}_t]
            +\mathfrak{Q}^{(2),k}_tB^k_x(\hat{\theta}_t)
            +\tilde{\E}[ \tilde{\hat{\mathfrak{Q}}}^{(2),k}_t B^k_\mu(\tilde{\theta}_t)(\hat{X}_t)]\Big)\\
            &+ \sum_{k=1}^d \Big(
            R^{k}_t C^k_\mu(\theta_t)(\hat{X}_t)
            +C^{k,*}_\mu(\hat{\theta}_t)(X_t)\hat{R}^{k}_t
            +C^{k,*}_x(\theta_t)\mathfrak{R}^{k}_t\\
            &+\tilde{\E}[C^{k,*}_\mu(\tilde{\theta}_t)(X_t)\tilde{\mathfrak{R}}^{k}_t]
            +\mathfrak{R}^{k}_t C^k_x(\hat{\theta}_t)
            +\tilde{\E}[ \tilde{\hat{\mathfrak{R}}}^{k}_t C^k_\mu(\tilde{\theta}_t)(\hat{X}_t)]\Big)\bigg)\dt\\
            &+ \mathfrak{Q}^{(1)}_t \dW^1_t
            +\mathfrak{Q}^{(2)}_t \,\mathrm{d}\hat{W}^1_t
            +\mathfrak{R}_t \dW^0_t,\\
            \mathfrak{P}_T=&\tilde{\E}[g_{\mu \mu}(\tilde{X}_T,\mu_T)( X_T,\hat{X}_T)] +g_{x\mu}(\hat{X}_T,\mu_T)(X_T)+g^*_{x\mu}(X_T,\mu_T)(\hat{X}_T),
        \end{aligned}\label{eq:Second-Level Second-Order Adjoint}
    \end{equation}
    \begin{remark}
        \begin{enumerate}[(i)]
            \item The above equation is of conditional McKean--Vlasov type, as the coefficients depend on the conditional law of the solution.
            \item Again, we note that the explicit dependence is on the joint law, i.e. on $\cL(\mathfrak{P}_t,X_t\mid \cF^0\vee \cF^1)$, $\cL(\mathfrak{P}_t,\hat{X}_t\mid \cF^0\vee \hat{\cF}^1)$ and $\cL(\mathfrak{P}_t,X_t,\hat{X}_t\mid \cF^0)$. The conditioning on $\cF^0$ comes from the fact that the original equation contains the conditional law $\cL(X_t\mid\cF^0_t)$, while the conditioning on $\cF^1$ is needed due to the structure of the Lions derivative and the interaction of $Y$ with $\hat{Y}$. This conditioning already appeared in the non-common noise setting (cf. \cite{spille2026}).
            \item The exact definition of a solution and its existence and uniqueness are discussed in Appendix \ref{appendix:Ex and Uni for Controlled MKVSDE with Common Noise}, with the final well-posedness result being Theorem \ref{theorem_appendix:existence and uniqueness of Cond-MKVBSDE}.
            \item The measurability of the coefficients was already discussed in Remark \ref{remark:after def of variational eq}.
        \end{enumerate}
    \end{remark}
    \begin{remark}
        We make some observations.
        \begin{enumerate}[(i)]
            \item If there is no $\mu$-dependence in the coefficients and the cost functions, then all affine terms and the terminal condition are equal to $0$, so the unique solution becomes $\mathfrak{P}\equiv 0$.\\
            Equally, if $C\equiv 0$, then the equation reduces to the one from \citet{spille2026}.
            \item $P$ takes values in the symmetric matrices (cf. Remark \ref{remark:after first two adjoints, existence, symmetry}). This makes sense, as $Y\otimes Y$ is also a symmetric matrix. On the contrary $Y\otimes \hat{Y}$ is only symmetric in the sense that
            \begin{equation*}
                (Y_t\otimes \hat{Y}_t)(\bar{\omega}_0,\bar{\omega}_1,\bar{\omega}_2)
                = (Y_t\otimes \hat{Y}_t)^*(\bar{\omega}_0,\bar{\omega}_2,\bar{\omega}_1).
            \end{equation*}
            In fact, $H_{\mu\mu}$ and $g_{\mu\mu}$ also have this symmetry (cf. again \cite{CarmonaDelarue2018I} Remark 5.89), and we can see that the other terms in \eqref{eq:Second-Level Second-Order Adjoint} also transform into each other if they are transposed and the second and third variable are swapped. Thus, due to uniqueness of a solution, the same symmetry holds for the second-level second-order adjoint, i.e.
            \begin{align*}
                \mathfrak{P}_t(\bar{\omega}_0,\bar{\omega}_1,\bar{\omega}_2)
                &=\mathfrak{P}^*_t(\bar{\omega}_0,\bar{\omega}_2,\bar{\omega}_1),\quad \\
                \mathfrak{Q}^{(1),k}(\bar{\omega}_0,\bar{\omega}_1,\bar{\omega}_2)
                &=\mathfrak{Q}^{(2),k,*}(\bar{\omega}_0,\bar{\omega}_2,\bar{\omega}_1)\quad\text{and}\\
                \mathfrak{R}^k_t(\bar{\omega}_0,\bar{\omega}_1,\bar{\omega}_2)
                &=\mathfrak{R}^{k,*}_t(\bar{\omega}_0,\bar{\omega}_2,\bar{\omega}_1).
            \end{align*}
            for all $k\in \{1,\dots,d\}$ and $\bar{\P}$-almost surely.
        \end{enumerate}
    \end{remark}

\section{The Duality Relations}\label{sec:The Duality Relations}
Since we desire to linearize the terms in the Taylor expansion of the cost functional, we will need the duality relations, which arise directly from Ito's formula. By Lemmas \ref{lemma:Estimates on Y and Z} and \ref{lemma:estimate on DeltaX-Y-Z}, we will put terms that are of higher-order into $o(\epsilon)$. First, we have the dualization of the first-order adjoint with the first-order variational process
\begin{equation}
    \begin{aligned}
        \bar{\E}\left[\langle p_T, Y_T\rangle\right]
        =& \bar{\E}\bigg[\int_0^T -f_x(\theta_t) Y_t
        - \tilde{\E}\left[f_\mu(\theta_t)(\tilde{X}_t) \tilde{Y}_t\right]\\
        &+\langle p_t,\delta A(t) \rangle 
        + \langle q_t,\delta B(t) \rangle
        + \langle r_t,\delta C(t) \rangle\dt\bigg],
    \end{aligned} \label{eq:duality p and Y}
\end{equation}
and the dualization of the first-order adjoint with the second-order variational process
\begin{equation}
    \begin{aligned}
        \bar{\E}\left[\langle p_T, Z_T\rangle\right]
        =& \bar{\E}\bigg[\int_0^T -f_x(\theta_t)Z_t
        - \tilde{\E}[f_\mu(\theta_t)(\tilde{X}_t) \tilde{Z}_t]  \\
        & + \big\langle p_t,  \frac{1}{2} A_{xx}(\theta_t)[Y_t,Y_t]
        +  \frac{1}{2} \tilde{\E}\left[ A_{y\mu}(\theta_t)(\tilde{X}_t)[ \tilde{Y}_t,\tilde{Y}_t]\right]\\ 
        &+\frac{1}{2}\tilde{\E} \Big[\tilde{\tilde\E}\big[ A_{\mu\mu}(\theta_t)(\tilde{X}_t,\tilde{\tilde{X}}_t)[\tilde{Y}_t,\tilde{\tilde{Y}}_t]\big]\Big] 
        +\tilde{\E}\left[ A_{x\mu}(\theta_t)(\tilde{X}_t)[ \tilde{Y}_t,Y_t]\right]
        \big\rangle \\
        &+ \big\langle q_t,  \frac{1}{2} B_{xx}(\theta_t)[Y_t,Y_t]
        +  \frac{1}{2} \tilde{\E}\left[ B_{y\mu}(\theta_t)(\tilde{X}_t)[ \tilde{Y}_t,\tilde{Y}_t]\right] \\
        &+\frac{1}{2}\tilde{\E} \Big[\tilde{\tilde\E}\big[ B_{\mu\mu}(\theta_t)(\tilde{X}_t,\tilde{\tilde{X}}_t)[\tilde{Y}_t,\tilde{\tilde{Y}}_t]\big]\Big] 
        +\tilde{\E}\left[ B_{x\mu}(\theta_t)(\tilde{X}_t)[ \tilde{Y}_t,Y_t]\right]
        \big\rangle \\
        &+ \big\langle r_t,  \frac{1}{2} C_{xx}(\theta_t)[Y_t,Y_t]
        +  \frac{1}{2} \tilde{\E}\left[ C_{y\mu}(\theta_t)(\tilde{X}_t)[ \tilde{Y}_t,\tilde{Y}_t]\right]\\
        &+\frac{1}{2}\tilde{\E} \left[\tilde{\tilde\E}\left[ C_{\mu\mu}(\theta_t)(\tilde{X}_t,\tilde{\tilde{X}}_t)[\tilde{Y}_t,\tilde{\tilde{Y}}_t]\right]\right]
        +\tilde{\E}\left[ C_{x\mu}(\theta_t)(\tilde{X}_t)[ \tilde{Y}_t,Y_t]\right]
        \big\rangle\dt\bigg] 
        +o(\epsilon).
    \end{aligned} \label{eq:duality p and Z}
\end{equation}
Towards the dualization of second-order terms, we look at the tensor product of the first-order variational process $Y$ with itself. Note again that $x,y\in\R^d$, it holds $x\otimes y:=xy^\top=xy^*\in \R^{d\times d}$, since we work in finite dimensions. We get
\begin{align*}
    &d\left(Y_t \otimes Y_t\right)\\
    &=\Big((A_x(\theta_t)Y_t+\tilde{\E}[A_\mu(\theta_t)(\tilde{X}_t)\tilde{Y}_t]+\delta A(t)) \otimes Y_t\\
    &+Y_t \otimes (A_x(\theta_t)Y_t+\tilde{\E}[A_\mu(\theta_t)(\tilde{X}_t)\tilde{Y}_t]+\delta A(t))\\
    &+(B_x(\theta_t)Y_t+\tilde{\E}[B_\mu(\theta_t)(\tilde{X}_t)\tilde{Y}_t]+\delta B(t))(B_x(\theta_t)Y_t+\tilde{\E}[B_\mu(\theta_t)(\tilde{X}_t)\tilde{Y}_t]+\delta B(t))^* \\
    &+(C_x(\theta_t)Y_t+\tilde{\E}[C_\mu(\theta_t)(\tilde{X}_t)\tilde{Y}_t]+\delta C(t))(C_x(\theta_t)Y_t+\tilde{\E}[C_\mu(\theta_t)(\tilde{X}_t)\tilde{Y}_t]+\delta C(t))^*\Big) \dt\\
    &+\sum_{k=1}^d\Big(\big(B^k_x(\theta_t)Y_t+\tilde{\E}[B^k_\mu(\theta_t)(\tilde{X}_t)\tilde{Y}_t]+\delta B^k(t) \big) \otimes Y_t\\
    &+Y_t \otimes \big(B^k_x(\theta_t)Y_t+\tilde{\E}[B^k_\mu(\theta_t)(\tilde{X}_t)\tilde{Y}_t]+\delta B^k(t) \big)\Big)\dW^{1,k}_t\\
    &+\sum_{k=1}^d\Big(\big(C^k_x(\theta_t)Y_t+\tilde{\E}[C^k_\mu(\theta_t)(\tilde{X}_t)\tilde{Y}_t]+\delta C^k(t) \big) \otimes Y_t\\
    &+Y_t \otimes \big(C^k_x(\theta_t)Y_t+\tilde{\E}[C^k_\mu(\theta_t)(\tilde{X}_t)\tilde{Y}_t]+\delta C^k(t) \big)\Big)\dW^{0,k}_t.
\end{align*}
Thus, the pairing of the first-level second-order adjoint process with this tensor product gives
\begin{equation}
    \begin{aligned}
        &\bar{\E}\left[\left\langle P_T, Y_T \otimes Y_T\right\rangle\right]\\
        =&\bar{\E}\bigg[\int_0^T-\langle H_{xx}(\theta_t)+ \tilde{\E}[ H_{y\mu}(\tilde{\theta}_t)(X_t) ]  , Y_t \otimes Y_t\rangle
        +\langle P_t, \tilde{\E}[A_\mu(\theta_t)(\tilde{X}_t)\tilde{Y}_t]\otimes Y_t\\
        &+Y_t \otimes \tilde{\E}[A_\mu(\theta_t)(\tilde{X}_t)\tilde{Y}_t]+\delta B(t)\delta B^*(t)
        +B_x(\theta_t)Y_t\tilde{\E}[B_\mu(\theta_t)(\tilde{X}_t)\tilde{Y}_t]^*\\
        &+\tilde{\E}[B_\mu(\theta_t)(\tilde{X}_t)\tilde{Y}_t]\tilde{\E}[B_\mu(\theta_t)(\tilde{X}_t)\tilde{Y}_t]^*
        +\tilde{\E}[B_\mu(\theta_t)(\tilde{X}_t)\tilde{Y}_t](B_x(\theta_t)Y_t)^*\\
        &+\delta C(t)\delta C^*(t)
        +C_x(\theta_t)Y_t\tilde{\E}[C_\mu(\theta_t)(\tilde{X}_t)\tilde{Y}_t]^*\\
        &+\tilde{\E}[C_\mu(\theta_t)(\tilde{X}_t)\tilde{Y}_t]\tilde{\E}[C_\mu(\theta_t)(\tilde{X}_t)\tilde{Y}_t]^*
        +\tilde{\E}[C_\mu(\theta_t)(\tilde{X}_t)\tilde{Y}_t](C_x(\theta_t)Y_t)^*\rangle \\
        & +\sum_{k=1}^d \langle Q^k_t, \tilde{\E}[B^k_\mu(\theta_t)(\tilde{X}_t)\tilde{Y}_t] \otimes Y_t+Y_t \otimes \tilde{\E}[B^k_\mu(\theta_t)(\tilde{X}_t)\tilde{Y}_t]\rangle \\
        & +\sum_{k=1}^d \langle R^k_t, \tilde{\E}[C^k_\mu(\theta_t)(\tilde{X}_t)\tilde{Y}_t] \otimes Y_t+Y_t \otimes \tilde{\E}[C^k_\mu(\theta_t)(\tilde{X}_t)\tilde{Y}_t]\rangle \dt\bigg]+o(\epsilon).
    \end{aligned}\label{eq:first second duality}
\end{equation}
For all the above dualizations it would have sufficed to look at the processes on the original probability space and use the standard expectation. For the dualization of the second-level second-order adjoint process the construction of the probability space $\bar{\Omega}$ is important, as the dualization needs to be with the tensor product of the first-order variational process $Y$ with its conditionally-independent copy $\hat{Y}$ (introduced above).
Note first that due to Remark \ref{remark_appendix:Conditional Copy of SDE},
\begin{equation*}
    \begin{aligned}
        d\hat{Y}_t 
        =& \left( A_x(\hat{\theta}_t)\hat{Y}_t +\tilde{\E}\left[A_\mu(\hat{\theta}_t)( \tilde{\hat{X}}_t) \tilde{\hat{Y}}_t\right] 
        + \delta \hat{A}(t) \right) \dt  \\
        &+ \left( B_x(\hat{\theta}_t)\hat{Y}_t + \tilde{\E}\left[B_\mu(\hat{\theta}_t)( \tilde{\hat{X}}_t) \tilde{\hat{Y}}_t\right]
        + \delta \hat{B}(t) \right) \,\mathrm{d}\hat{W}^1_t \\
        &+ \left( C_x(\hat{\theta}_t)\hat{Y}_t + \tilde{\E}\left[C_\mu(\hat{\theta}_t)( \tilde{\hat{X}}_t) \tilde{\hat{Y}}_t\right]
        + \delta \hat{C}(t) \right) \dW^0_t  \\
        \hat{Y}_0 =& 0.
    \end{aligned}
\end{equation*}
This results in
\begin{align*}
    &d\left(Y_t \otimes \hat{Y}_t\right)\\
    =&\bigg((A_x(\theta_t)Y_t+\tilde{\E}[A_\mu(\theta_t)(\tilde{X}_t)\tilde{Y}_t]+\delta A(t)) \otimes \hat{Y}_t
    +Y_t \otimes (A_x(\hat{\theta}_t)\hat{Y}_t+\tilde{\E}[A_\mu(\hat{\theta}_t)(\tilde{X}_t)\tilde{Y}_t]+\delta \hat{A}(t))\\
    &+\left( C_x(\theta_t)Y_t + \tilde{\E}\left[C_\mu(\theta_t)(\tilde{X}_t)\tilde{Y}_t\right]+ \delta C(t) \right)
    \left( C_x(\hat{\theta}_t)\hat{Y}_t + \tilde{\E}\left[C_\mu(\hat{\theta}_t)( \tilde{\hat{X}}_t) \tilde{\hat{Y}}_t\right]+ \delta \hat{C}(t) \right)^*\bigg)\dt\\
    &+\sum_{k=1}^d\left(B^k_x(\theta_t)Y_t+\tilde{\E}[B^k_\mu(\theta_t)(\tilde{X}_t)\tilde{Y}_t]+\delta B^k(t) \right) \otimes \hat{Y}_t \dW^{1,k}_t\\
    &+\sum_{k=1}^d Y_t \otimes \left(B^k_x(\hat{\theta}_t)\hat{Y}_t+\tilde{\E}[B^k_\mu(\hat{\theta}_t)(\tilde{X}_t)\tilde{Y}_t]+\delta \hat{B}^k(t) \right)\,\mathrm{d}\hat{W}^{1,k}_t,\\
    &+\bigg(\sum_{k=1}^d\left(C^k_x(\theta_t)Y_t+\tilde{\E}[C^k_\mu(\theta_t)(\tilde{X}_t)\tilde{Y}_t]+\delta C^k(t) \right) \otimes \hat{Y}_t\\
    &+\sum_{k=1}^d Y_t \otimes \left(C^k_x(\hat{\theta}_t)\hat{Y}_t+\tilde{\E}[C^k_\mu(\hat{\theta}_t)(\tilde{X}_t)\tilde{Y}_t]+\delta \hat{C}^k(t) \right)\bigg)\dW^{0,k}_t,
\end{align*}
and thus,
\begin{equation}
    \begin{aligned}
        &\bar{\E}\left[\left\langle \mathfrak{P}_T, Y_T \otimes \hat{Y}_T\right\rangle\right]\\
        =&\bar{\E}\bigg[\int_0^T-\big\langle \tilde{\E} \left[H_{\mu\mu}(\tilde{\theta}_t)(X_t,\hat{X}_t)\right]
        +2 H_{x\mu}(\hat{\theta}_t)(X_t)
        + P^* _t A_\mu(\theta_t)(\hat{X}_t)
        + P_t A_\mu(\theta_t)(\hat{X}_t)\\
        & + B^*_x(\theta_t) P_t B_\mu(\theta_t)(\hat{X}_t) 
        + \tilde{\E}[B^*_\mu(\tilde{\theta}_t)(X_t)\tilde{P}_t B_\mu(\tilde{\theta}_t)(\hat{X}_t)]
        + B^*_x(\theta_t) P^*_t B_\mu(\theta_t)(\hat{X}_t) \\
        & + C^*_x(\theta_t) P_t C_\mu(\theta_t)(\hat{X}_t) 
        + \tilde{\E}[C^*_\mu(\tilde{\theta}_t)(X_t)\tilde{P}_t C_\mu(\tilde{\theta}_t)(\hat{X}_t)]
        + C^*_x(\theta_t) P^*_t C_\mu(\theta_t)(\hat{X}_t) \\
        &+ \sum_{k=1}^d \Big(
        Q^{k,*}_t B^k_\mu(\theta_t)(\hat{X}_t)
        +Q^{k}_t B^k_\mu(\theta_t)(\hat{X}_t)\\
        &+R^{k,*}_t C^k_\mu(\theta_t)(\hat{X}_t)
        +R^{k}_t C^k_\mu(\theta_t)(\hat{X}_t), Y_t \otimes \hat{Y}_t\big\rangle+\langle \mathfrak{P}_t ,\delta C(t)\delta \hat{C}^*(t)\rangle \dt\bigg] +o(\epsilon),
    \end{aligned}\label{eq:second second duality}
\end{equation}
where we used Fubini for the cancellation of the terms.
\begin{remark}
    Note the appearance of the $\langle P_t,\delta B(t) \delta B^*(t)+\delta C(t) \delta C^*(t)\rangle$-term in the dualization \eqref{eq:first second duality}. This term comes from the quadratic variation of $Y$ and neither is of high order, nor can be eliminated by the dualization as is does not contain $Y$. Therefore, this term and, in particular, the first-level second-order adjoint process, will appear in the maximum principle. On the other hand, the quadratic covariation of $Y$ and $\hat{Y}$ only comes from the common noise, as the stochastic integrals with respect to $W^1$ and $\hat{W}^1$ are by construction independent. Thus, in the dualization \eqref{eq:second second duality} only the term $\langle \mathfrak{P}_t ,\delta C(t)\delta \hat{C}^*(t)\rangle$ appears. Still, this will lead to the appearance of the second-level second-order adjoint process in the maximum principle. This is a key difference to the non-common noise case \cite{spille2026}, where the second-level second-order process does not appear, as $Y$ and the lift $\hat{Y}$ are fully independent, so no quadratic covariation appears.
\end{remark}
\section{The Taylor Expansion of the Cost Functional and the Maximum Principle}\label{sec:Expansion of Cost Functional and Maximum Principle}
    We come to the most important result towards the maximum principle, the expansion of the cost functional up to order $o(\epsilon)$.
    \begin{proposition} \label{proposition:expansion of cost functional}
        Under Assumption \ref{Assumption:C^{2,1}}, if $\alpha\in \A$ is optimal, it holds
        \begin{equation}
            J(\alpha^{\epsilon})-J(\alpha)
            =  \bar{\E} \left[ \int_0^T \delta H(t)
            +\frac{1}{2}\big\langle P_t,
            \delta B(t)\delta B(t)^*
            +\delta C(t)\delta C^*(t)\big\rangle
            +\frac{1}{2}\langle \mathfrak{P}_t ,\delta C(t)\delta \hat{C}^*(t)\rangle \dt\right] +o(\epsilon).\label{eq:Final Expansion of Cost Functional}
        \end{equation}
    \end{proposition}
    \begin{proof}
        The proof follows the same line as in \cite{spille2026}, but also the new terms coming from the common noise are treated.\\
        Using Taylor's formula (cf. Lemma \ref{lemma:Taylor formula for x and mu dependence}), we get the expansion of the cost functional
        \begin{align*}
            &J(\alpha^{\epsilon})-J(\alpha )\\
            =&\mathbb{E}\Big[\int_0^T f\left(\theta_t^\epsilon\right)-f\left(\theta_t\right) \dt
            +g(X_T^{\epsilon}, \mu^\epsilon_T)-g(X_T, \mu_T)\Big] \\
            = & \bar{\E} \Bigg[ \int_0^T f_x(\theta_t)[\Delta X_t]+\tilde{\mathbb{E}}\left[f_\mu(\theta_t)(\tilde{X}_t)[\Delta \tilde{X}_t\right]
            +\delta f(t)\\
            &+\frac{1}{2}f_{xx}(\theta_t)[\Delta X_t,\Delta X_t]
            +\tilde{\mathbb{E}}\left[f_{x\mu}(\theta_t)(\tilde{X}_t)[\Delta \tilde{X}_t,\Delta X_t]\right]\\
            & +\frac{1}{2}\tilde{\mathbb{E}}\left[f_{y\mu}(\theta_t)(\tilde{X}_t)[\Delta \tilde{X}_t,\Delta \tilde{X}_t]\right]
            +\frac{1}{2}\tilde{\tilde{\E}} \left[\tilde{\mathbb{E}} \left[f_{\mu\mu}(\theta_t)(\tilde{X}_t,\tilde{\tilde{X}}_t)[\Delta \tilde{X}_t,\Delta \tilde{\tilde{X}}_t]\right]\right] \dt\\
            & +g_x(X_T,\mu_T)[\Delta X_T]
            +\tilde{\mathbb{E}}\left[g_\mu(X_T,\mu_T)(\tilde{X}_T)[\Delta \tilde{X}_T]\right] \\
            & +\frac{1}{2}g_{x x}(X_T,\mu_T)[\Delta X_T,\Delta X_T]
            +\frac{1}{2}\tilde{\mathbb{E}}\left[g_{y\mu}(X_T,\mu_T)(\tilde{X}_T)[\Delta \tilde{X}_T,\Delta \tilde{X}_T]\right]\\
            & +\tilde{\mathbb{E}}\left[g_{x\mu}(X_T,\mu_T)(\tilde{X}_T)[\Delta \tilde{X}_T,\Delta X_T]\right]\\
            &+\frac{1}{2}\tilde{\tilde{\E}} \left[\tilde{\mathbb{E}} \left[g_{\mu\mu}(X_T,\mu_T)(\tilde{X}_T,\tilde{\tilde{X}}_T)[\Delta \tilde{X}_T,\Delta \tilde{\tilde{X}}_T]\right]\right]\Bigg]
            +o(\epsilon) .
        \end{align*}
        Note that the Taylor expansion was done for fixed $\alpha^\epsilon$, but we can afterwards replace $\alpha^\epsilon$ by $\alpha$ as the resulting remainders are at least of order $\epsilon \sup_t\E[ \|\Delta X_t\|]\in o(\epsilon)$.\\ 
        By the linearity, we can now everywhere replace $\Delta X_t$ by $Y_t+Z_t$ and the difference is in $o(\epsilon)$ by the estimate from Lemma \ref{lemma:estimate on DeltaX-Y-Z}. Furthermore, also the terms containing $[Y ,Z]$ or $[Z,Z]$ (the quadratic terms from the second derivatives) are of high-enough order by the estimates from Lemma \ref{lemma:Estimates on Y and Z}, e.g. 
        \begin{equation*}
            \E\left[\frac{1}{2}f_{xx}(\theta_t)[Y_t+Z_t,Y_t+Z_t]\right]
            =\E\left[\frac{1}{2}f_{xx}(\theta_t)[Y_t,Y_t]\right]+o(\epsilon),
        \end{equation*}
        and similarly for the other terms. Thus,
        \begin{align}
            J(\alpha^{\epsilon})-J(\alpha )
            = & \bar{\E}\Bigg[ \int_0^T f_x(\theta_t)[Y_t+Z_t]+\tilde{\mathbb{E}}\left[f_\mu(\theta_t)(\tilde{X}_t)(\tilde{Y}_t+\tilde{Z}_t)\right]
            +\delta f(t)\nonumber\\
            &+\frac{1}{2}f_{xx}(\theta_t)[Y_t,Y_t]
            +\tilde{\mathbb{E}}\left[f_{x\mu}(\theta_t)[\tilde{X}_t)[\tilde{Y}_t,Y_t]\right]\nonumber\\
            & +\frac{1}{2}\tilde{\mathbb{E}}\left[f_{y\mu}(\theta_t)(\tilde{X}_t)[\tilde{Y}_t,\tilde{Y}_t]\right]
            +\frac{1}{2}\tilde{\tilde{\E}} \left[\tilde{\mathbb{E}} \left[f_{\mu\mu}(\theta_t)(\tilde{X}_t,\tilde{\tilde{X}}_t)[\tilde{Y}_t,\tilde{\tilde{Y}}_t]\right]\right] \dt\nonumber\\
            & +g_x(X_T,\mu_T)[Y_T+Z_T]
            +\tilde{\mathbb{E}}\left[g_\mu(X_T,\mu_T)(\tilde{X}_T)[\tilde{Y}_T+\tilde{Z}_T]\right] \label{eq:First Duality Replacement}\\
            & +\frac{1}{2}g_{x x}(X_T,\mu_T)[Y_T,Y_T]
            +\frac{1}{2}\tilde{\mathbb{E}}\left[g_{y\mu}(X_T,\mu_T)(\tilde{X}_T)[\tilde{Y}_T,\tilde{Y}_T]\right] \nonumber
            \\
            &+\tilde{\mathbb{E}}\left[g_{x\mu}(X_T,\mu_T)(\tilde{X}_T)[\tilde{Y}_T,Y_T]\right] \nonumber\\
            &+\frac{1}{2}\tilde{\tilde{\E}} \left[\tilde{\mathbb{E}} \left[g_{\mu\mu}(X_T,\mu_T)(\tilde{X}_T,\tilde{\tilde{X}}_T)[\tilde{Y}_T,\tilde{\tilde{Y}}_T]\right]\right]\Bigg]+o(\epsilon).\nonumber
        \end{align}
        Now, we replace \eqref{eq:First Duality Replacement} with the duality relations \eqref{eq:duality p and Y} and \eqref{eq:duality p and Z}, which gives
        \begin{align}
            J(\alpha^{\epsilon})-J(\alpha )
            = & \bar{\E} \Bigg[ \int_0^T \delta H(t)
            +\frac{1}{2}H_{xx}(\theta_t)[Y_t,Y_t]
            +\tilde{\mathbb{E}}\left[H_{x\mu}(\theta_t)[\tilde{X}_t)[\tilde{Y}_t,Y_t]\right]\nonumber\\
            & +\frac{1}{2}\tilde{\mathbb{E}}\left[H_{y\mu}(\theta_t)(\tilde{X}_t)[\tilde{Y}_t,\tilde{Y}_t]\right]
            +\frac{1}{2}\tilde{\tilde{\E}} \left[\tilde{\mathbb{E}} \left[H_{\mu\mu}(\theta_t)(\tilde{X}_t,\tilde{\tilde{X}}_t)[\tilde{Y}_t,\tilde{\tilde{Y}}_t]\right]\right] \dt\nonumber\\
            & +\frac{1}{2}g_{x x}(X_T,\mu_T)[Y_T,Y_T]
            +\frac{1}{2}\tilde{\mathbb{E}}\left[g_{y\mu}(X_T,\mu_T)(\tilde{X}_T)[\tilde{Y}_T,\tilde{Y}_T]\right] \label{eq:Second Duality Replacement}
            \\
            &+\tilde{\mathbb{E}}\left[g_{x\mu}(X_T,\mu_T)(\tilde{X}_T)[\tilde{Y}_T,Y_T]\right] \nonumber\\
            &+\frac{1}{2}\tilde{\tilde{\E}} \left[\tilde{\mathbb{E}} \left[g_{\mu\mu}(X_T,\mu_T)(\tilde{X}_T,\tilde{\tilde{X}}_T)[\tilde{Y}_T,\tilde{\tilde{Y}}_T]\right]\right]\Bigg]
            +o(\epsilon),\nonumber
        \end{align}
        and then replace
        \eqref{eq:Second Duality Replacement} with the duality relation \eqref{eq:first second duality} to get
        \begin{align}
            &J(\alpha^{\epsilon})-J(\alpha )\nonumber\\
            = & \bar{\E}\Bigg[ \int_0^T \delta H(t)
            +\tilde{\mathbb{E}}\left[H_{x\mu}(\theta_t)[\tilde{X}_t)[\tilde{Y}_t,Y_t]\right]
            +\frac{1}{2}\tilde{\tilde{\E}} \left[\tilde{\mathbb{E}} \left[H_{\mu\mu}(\theta_t)(\tilde{X}_t,\tilde{\tilde{X}}_t)[\tilde{Y}_t,\tilde{\tilde{Y}}_t]\right]\right] \nonumber\\
            &+\frac{1}{2}\langle P_t, \tilde{\E}[A_\mu(\theta_t)(\tilde{X}_t)\tilde{Y}_t]\otimes Y_t \nonumber\\
            &+Y_t \otimes \tilde{\E}[A_\mu(\theta_t)(\tilde{X}_t)\tilde{Y}_t]+\delta B(t)\delta B^*(t)
            +B_x(\theta_t)Y_t\tilde{\E}[B_\mu(\theta_t)(\tilde{X}_t)\tilde{Y}_t]^* \nonumber\\
            &+\tilde{\E}[B_\mu(\theta_t)(\tilde{X}_t)\tilde{Y}_t]\tilde{\E}[B_\mu(\theta_t)(\tilde{X}_t)\tilde{Y}_t]^*
            +\tilde{\E}[B_\mu(\theta_t)(\tilde{X}_t)\tilde{Y}_t](B_x(\theta_t)Y_t)^* \nonumber\\
            &+\delta C(t)\delta C^*(t)
            +C_x(\theta_t)Y_t\tilde{\E}[C_\mu(\theta_t)(\tilde{X}_t)\tilde{Y}_t]^* \nonumber\\
            &+\tilde{\E}[C_\mu(\theta_t)(\tilde{X}_t)\tilde{Y}_t]\tilde{\E}[C_\mu(\theta_t)(\tilde{X}_t)\tilde{Y}_t]^*
            +\tilde{\E}[C_\mu(\theta_t)(\tilde{X}_t)\tilde{Y}_t](C_x(\theta_t)Y_t)^*\rangle \nonumber\\
            & +\frac{1}{2}\sum_{k=1}^d \langle Q^k_t, \tilde{\E}[B^k_\mu(\theta_t)(\tilde{X}_t)\tilde{Y}_t] \otimes Y_t+Y_t \otimes \tilde{\E}[B^k_\mu(\theta_t)(\tilde{X}_t)\tilde{Y}_t]\rangle \nonumber\\
            & +\frac{1}{2}\sum_{k=1}^d \langle R^k_t, \tilde{\E}[C^k_\mu(\theta_t)(\tilde{X}_t)\tilde{Y}_t] \otimes Y_t+Y_t \otimes \tilde{\E}[C^k_\mu(\theta_t)(\tilde{X}_t)\tilde{Y}_t]\rangle \Bigg)  \dt \nonumber \\
            &+\tilde{\mathbb{E}}\left[g_{x\mu}(X_T,\mu_T)(\tilde{X}_T)[\tilde{Y}_T,Y_T]\right] 
            +\frac{1}{2}\tilde{\tilde{\E}} \left[\tilde{\mathbb{E}} \left[g_{\mu\mu}(X_T,\mu_T)(\tilde{X}_T,\tilde{\tilde{X}}_T)[\tilde{Y}_T,\tilde{\tilde{Y}}_T]\right]\right]\Bigg] \label{eq:Third Duality Replacement}
            \\
            &+o(\epsilon).\nonumber
        \end{align}
        Finally, we use the second-level second-order adjoint process and its dualization \eqref{eq:second second duality} to replace \eqref{eq:Third Duality Replacement} and arrive at
        \begin{align*}
            &J(\alpha^{\epsilon})-J(\alpha )\\
            = & \bar{\E}\Bigg[ \int_0^T \delta H(t)
            +\tilde{\mathbb{E}}\left[H_{x\mu}(\theta_t)[\tilde{X}_t)[\tilde{Y}_t,Y_t]\right]
            +\frac{1}{2}\tilde{\tilde{\E}} \left[\tilde{\mathbb{E}} \left[H_{\mu\mu}(\theta_t)(\tilde{X}_t,\tilde{\tilde{X}}_t)[\tilde{Y}_t,\tilde{\tilde{Y}}_t]\right]\right] \nonumber\\
            &+\frac{1}{2}\langle P_t, \tilde{\E}[A_\mu(\theta_t)(\tilde{X}_t)\tilde{Y}_t]\otimes Y_t \nonumber\\
            &+Y_t \otimes \tilde{\E}[A_\mu(\theta_t)(\tilde{X}_t)\tilde{Y}_t]+\delta B(t)\delta B^*(t)
            +B_x(\theta_t)Y_t\tilde{\E}[B_\mu(\theta_t)(\tilde{X}_t)\tilde{Y}_t]^* \nonumber\\
            &+\tilde{\E}[B_\mu(\theta_t)(\tilde{X}_t)\tilde{Y}_t]\tilde{\E}[B_\mu(\theta_t)(\tilde{X}_t)\tilde{Y}_t]^*
            +\tilde{\E}[B_\mu(\theta_t)(\tilde{X}_t)\tilde{Y}_t](B_x(\theta_t)Y_t)^* \nonumber\\
            &+\delta C(t)\delta C^*(t)
            +C_x(\theta_t)Y_t\tilde{\E}[C_\mu(\theta_t)(\tilde{X}_t)\tilde{Y}_t]^* \nonumber\\
            &+\tilde{\E}[C_\mu(\theta_t)(\tilde{X}_t)\tilde{Y}_t]\tilde{\E}[C_\mu(\theta_t)(\tilde{X}_t)\tilde{Y}_t]^*
            +\tilde{\E}[C_\mu(\theta_t)(\tilde{X}_t)\tilde{Y}_t](C_x(\theta_t)Y_t)^*\rangle \nonumber\\
            & +\frac{1}{2}\sum_{k=1}^d \langle Q^k_t, \tilde{\E}[B^k_\mu(\theta_t)(\tilde{X}_t)\tilde{Y}_t] \otimes Y_t+Y_t \otimes \tilde{\E}[B^k_\mu(\theta_t)(\tilde{X}_t)\tilde{Y}_t]\rangle \nonumber\\
            & +\frac{1}{2}\sum_{k=1}^d \langle R^k_t, \tilde{\E}[C^k_\mu(\theta_t)(\tilde{X}_t)\tilde{Y}_t] \otimes Y_t+Y_t \otimes \tilde{\E}[C^k_\mu(\theta_t)(\tilde{X}_t)\tilde{Y}_t]\rangle \Bigg)  \dt \nonumber \\
            &+\frac{1}{2}\bigg(-\big\langle \tilde{\E} \left[H_{\mu\mu}(\tilde{\theta}_t)(X_t,\hat{X}_t)\right]
            +2 H_{x\mu}(\hat{\theta}_t)(X_t)
            + P^* _t A_\mu(\theta_t)(\hat{X}_t)
            + P_t A_\mu(\theta_t)(\hat{X}_t)\\
            & + B^*_x(\theta_t) P_t B_\mu(\theta_t)(\hat{X}_t) 
            + \tilde{\E}[B^*_\mu(\tilde{\theta}_t)(X_t)\tilde{P}_t B_\mu(\tilde{\theta}_t)(\hat{X}_t)]
            + B^*_x(\theta_t) P^*_t B_\mu(\theta_t)(\hat{X}_t) \\
            & + C^*_x(\theta_t) P_t C_\mu(\theta_t)(\hat{X}_t) 
            + \tilde{\E}[C^*_\mu(\tilde{\theta}_t)(X_t)\tilde{P}_t C_\mu(\tilde{\theta}_t)(\hat{X}_t)]
            + C^*_x(\theta_t) P^*_t C_\mu(\theta_t)(\hat{X}_t) \\
            &+ \sum_{k=1}^d \Big(
            Q^{k,*}_t B^k_\mu(\theta_t)(\hat{X}_t)
            +Q^{k}_t B^k_\mu(\theta_t)(\hat{X}_t)\\
            &+R^{k,*}_t C^k_\mu(\theta_t)(\hat{X}_t)
            +R^{k}_t C^k_\mu(\theta_t)(\hat{X}_t), Y_t \otimes \hat{Y}_t\big\rangle+\langle \mathfrak{P}_t ,\delta C(t)\delta \hat{C}^*(t)\rangle\bigg)\Bigg]+o(\epsilon)\\
            = & \bar{\E} \left[ \int_0^T \delta H(t)
            +\frac{1}{2}\big\langle P_t,
            \delta B(t)\delta B(t)^*
            +\delta C(t)\delta C^*(t)\big\rangle
            +\frac{1}{2}\langle \mathfrak{P}_t ,\delta C(t)\delta \hat{C}^*(t)\rangle \dt\right] +o(\epsilon),
        \end{align*}
        giving the result.
    \end{proof}\noindent
    \noindent
    For fixed $u\in U$ define 
    \begin{equation*}
        \delta \varphi^u_t:=\varphi(t,X_t,\mu_t,u,p_t,q_t,r_t) - \varphi(t,X_t,\mu_t,\alpha_t,p_t,q_t,r_t)
    \end{equation*}
    for $\varphi=H,B,C$.
    We also recall the notation from Section \ref{subsec:Lifting and Third Adjoint}, that we use a copy $\hat{\Omega}$ of the space $\Omega^1$ where $W^1$ is defined and work on the product space $\bar{\Omega}=\Omega^0\times\Omega^1\times \hat{\Omega}$, which allows to independently lift only 'the randomness coming from $W^1$/$\Omega^1$' but keep 'the randomness coming from $W^0$/$\Omega^0$', i.e. we can define, for a random variable $\Psi:\Omega^0\times\Omega^1\to \R^m$, a (conditionally on $W^0$)-independent copy, by denoting
    \begin{equation*}
        \Psi(\bar{\omega})=\Psi(\bar{\omega}_0,\bar{\omega}_1,\bar{\omega}_2):=\Psi(\bar{\omega}_0,\bar{\omega}_1),\quad \text{and}\quad \hat{\Psi}(\bar{\omega})=\hat{\Psi}(\bar{\omega}_0,\bar{\omega}_1,\bar{\omega}_2):=\Psi(\bar{\omega}_0,\bar{\omega}_2).
    \end{equation*}
    Using \eqref{eq:Final Expansion of Cost Functional}, we get the following maximum principle.
    \begin{theorem}[Peng's Maximum Principle] \label{Theorem:maximum principle}
        Under Assumption \ref{Assumption:C^{2,1}}, if $\alpha$ is optimal and $u\in U$, then for Lebesgue almost every $t\in[0,T]$ and $\P^0$-almost surely,
        \begin{align*}
            0
            \leq  & \E^1\Bigg[\hat{\E}\bigg[
            \delta H^u_t 
            + \frac{1}{2} \langle P_t,\delta B^u_t(\delta B^u_t)^*\rangle
            + \frac{1}{2} \langle P_t,\delta C^u_t(\delta C^u_t)^*\rangle
            + \frac{1}{2} \langle \mathfrak{P}_t,\delta C^u_t(\delta \hat{C}^u_t)^*\rangle\bigg]\Bigg],
        \end{align*}
        or equivalently for almost all $t\in [0,T]$ and $\bar{\P}$-almost surely
        \begin{align*}
            0
            \leq  & \bar{\E}\bigg[
            \delta H^u_t
            + \frac{1}{2} \langle P_t,\delta B^u_t(\delta B^u_t)^*\rangle
            + \frac{1}{2} \langle P_t,\delta C^u_t(\delta C^u_t)^*\rangle
            + \frac{1}{2} \langle \mathfrak{P}_t,\delta C^u_t(\delta \hat{C}^u_t)^*\rangle\,\bigg|\, (W^0_{s})_{s\in[0,t]}\bigg],
        \end{align*}
        where $(p,q,r)$ (contained in $H$), $P$ and $\mathfrak{P}$ come from the first \eqref{eq:First order adjoint}, the second \eqref{eq:First-Level second-order adjoint equation} and the third \eqref{eq:Second-Level Second-Order Adjoint} adjoint equation, respectively.
    \end{theorem}
    \begin{proof}
        For $t_0\in[0,T]$, $A=A_0\times A_1\in \cF_{t_0}$ and $E_\epsilon=[t_0,t_0+\epsilon]$ choose the spike variation
        \begin{equation*}
            \alpha_t^{\epsilon}(\omega) =\begin{cases}
            u,&\text{for } t \in [t_0,t_0+\epsilon],\,\omega\in A, \\
            \alpha_t(\omega),&\text{else}  .
            \end{cases}
        \end{equation*}
        Therefore the above Proposition \ref{proposition:expansion of cost functional} gives, due to the optimality of $\alpha$,
        \begin{align*}
            0
            \leq & J(\alpha^{\epsilon})-J(\alpha)\\
            = & \bar{\E} \left[ \mathbf{1}_A \int_{t_0}^{t_0+\epsilon} \delta H^u_t
            +\frac{1}{2}\big\langle P_t,
            \delta B^u_t(\delta B^u_t)^*
            +\delta C^u_t(\delta C^u_t)^*\big\rangle
            +\widehat{\mathbf{1}_{A_1}}\frac{1}{2}\langle \mathfrak{P}_t ,\delta C^u_t(\delta \hat{C}^u_t)^*\rangle\dt\right] +o(\epsilon).
        \end{align*}
        Dividing by $\epsilon$ and letting $\epsilon\to 0$ gives, by Lebesgue differentiation, for almost every $t_0\in[0,T]$
        \begin{align*}
            0
            \leq & \bar{\E} \left[  \mathbf{1}_A \left(\delta H^u_{t_0}
            +\frac{1}{2}\big\langle P_{t_0},
            \delta B^u_{t_0}(\delta B^u_{t_0})^*
            +\delta C^u_{t_0}(\delta C^u_{t_0})^*\big\rangle
            +\widehat{\mathbf{1}_{A_1}}\frac{1}{2}\langle \mathfrak{P}_{t_0} ,\delta C^u_{t_0}(\delta \hat{C}^u_{t_0})^*\rangle\right)\right] .
        \end{align*}
        Now, using Fubini to split up the integrals over $\Omega^0,\Omega^1$ and $\hat{\Omega}$, we get
        \begin{align*}
            0
            \leq  & \E^0 \Bigg[  \mathbf{1}_{A_0} \E^1\bigg[\hat{\E}\Big[\mathbf{1}_{A_1}\big(\delta H^u_{t_0}
            +\frac{1}{2}\big\langle P_{t_0},
            \delta B^u_{t_0}(\delta B^u_{t_0})^*
            +\delta C^u_{t_0}(\delta C^u_{t_0})^*\big\rangle
            +\widehat{\mathbf{1}_{A_1}}\frac{1}{2}\langle \mathfrak{P}_{t_0} ,\delta C^u_{t_0}(\delta \hat{C}^u_{t_0})^*\rangle\big)\Big]\bigg]\Bigg] .
        \end{align*}
        Thus, $\P^0$-almost surely
        \begin{align*}
            0
            \leq  & \E^1\left[\hat{\E}\left[\mathbf{1}_{A_1}\left(\delta H^u_{t_0}
            +\frac{1}{2}\big\langle P_{t_0},
            \delta B^u_{t_0}(\delta B^u_{t_0})^*
            +\delta C^u_{t_0}(\delta C^u_{t_0})^*\big\rangle
            +\widehat{\mathbf{1}_{A_1}}\frac{1}{2}\langle \mathfrak{P}_{t_0} ,\delta C^u_{t_0}(\delta \hat{C}^u_{t_0})^*\rangle\right)\right]\right].
        \end{align*}
        Choosing $A_1=\Omega^1$ gives the result.
    \end{proof}
    \begin{remark}[Relation with the dynamic programming approach] \label{remark:relationToDPP}
        In the classical Markovian stochastic control problem given a controlled diffusion
        \begin{equation*}
            \mathrm{d}X^{x,t}_s=A(s,X_s,\alpha_s)\ds+B(s,X_s,\alpha_s)\dW_s,\quad X_s=x
        \end{equation*}
        the stochastic maximum principle and the dynamic programming principle are formally connected through
        the derivatives of the value function. More precisely, define
        \begin{equation*}
            J(t,x;\alpha):=\mathbb{E}\left[
            \int_t^T f(s,X^{x,t}_s,\alpha_s)\ds
            + g(X_T^{x,t})
            \right]
        \end{equation*}
        and the value function $V(t,x):=\min_{\alpha\in\A} J(t,x;\alpha)$. If $V$ is sufficiently smooth, $X$ is optimal and $\alpha$ is an optimal control, one has (up to sign convention)
        \begin{align*}
            -V_t(t,X_t)
            &=
            \cH_{\mathrm{cl}} \left(t,X_t,\alpha_t,V_x(t,X_t),V_{xx}(t,X_t)\right)
            \\
            &= \min_{u\in U} \cH_{\mathrm{cl}} \left(t,X_t,u,V_x(t,X_t),V_{xx}(t,X_t)\right),
        \end{align*}
        for a.e. $t\in[0,T]$, almost surely, where the classical extended Hamiltonian is given by
        \begin{align*}
            &\cH_{\mathrm{cl}}(t,x,u,p,P) 
            :=\frac{1}{2}\langle P,B(t,x,u)B^*(t,x,u)\rangle
            +\langle p,A(t,x,u)\rangle
            +f(t,x,u).
        \end{align*}
        Further, the first and second adjoint processes can be identified with
        \begin{equation*}
            p_t = V_x(t,X_t),\qquad
            P_t = V_{xx}(t,X_t).
        \end{equation*}
        This is the usual formal bridge between Peng's stochastic maximum principle and the HJB equation (cf. \citet{Yong1999} Chapter 4).\\
        In the present conditional McKean--Vlasov setting (leaving out the time-dependence of the coefficients as in \cite{PhamWei2017DPPforCommonNoise}) the state variable in the dynamic programming formulation is the conditional law
        \begin{equation*}
            \mu_t = \cL(X_t\mid \cF^0_t).
        \end{equation*}
        Consequently, the corresponding value function $v=v(t,\mu)$ is a function on the Wasserstein
        space with Bellman equation (cf. \cite{PhamWei2017DPPforCommonNoise} equation (4.7))
        \begin{align*}
            -v_t(t,\mu)
            =&\inf_{a\in U}
            \left(\widehat f(\mu,a)
            +\mu\left(L^a v(t,\mu)\right)
            +\mu\otimes\mu\left(M^a v(t,\mu)\right)\right),\\
            v(T,\mu)=&\widehat g(\mu),
        \end{align*}
        where
        \begin{equation*}
            \widehat f(\mu,a)
            :=\int_{\mathbb R^d} f(x,\mu,a)\,\mu(\mathrm{d}x),\quad
            \widehat g(\mu)
            :=\int_{\mathbb R^d} g(x,\mu)\,\mu(\mathrm{d}x),
        \end{equation*}
        and
        \begin{align*}
            L^a\varphi(\mu)(x)
            &:=
            \left\langle\varphi_\mu(\mu)(x), A(x,\mu,a) \right\rangle
            +\frac{1}{2}\langle \varphi_{y\mu}(\mu)(x), \left(BB^*+CC^*\right)(x,\mu,a)\rangle 
        \end{align*}
        and
        \begin{equation*}
            M^a\varphi(\mu)(x,x')
            :=\frac{1}{2}\langle \varphi_{\mu\mu}(\mu)(x,x'), C(x,\mu,a)C^*(x',\mu,a)\rangle.
        \end{equation*}       
        Defining
        \begin{align*}
            \cH(t,\mu,a,\varphi,\Phi,\Psi)
            :=&\int_{\mathbb R^d} \frac{1}{2}\langle \Phi, \left( BB^*+CC^* \right)(x,\mu,a)\rangle 
            +\left\langle \varphi,A(x,\mu,a)\right\rangle+f(x,\mu,a)\,\mu(\mathrm{d}x)\\
            &+\frac{1}{2}\int_{\mathbb R^d}\int_{\mathbb R^d} \langle \Psi, C(x,\mu,a)C^*(x',\mu,a)\rangle \,\mu(\mathrm{d}x)\mu(\mathrm{d}x'),
        \end{align*}
        the above equation can be rewritten as
        \begin{align*}
            -v_t(t,\mu)
            &=\inf_{a\in U} \cH\left(t,\mu,a,v_\mu(t,\mu),v_{y\mu}(t,\mu),v_{\mu\mu}(t,\mu)\right),\\
            v(T,\mu)&=\widehat g(\mu).
        \end{align*}
        Therefore, we expect, for an optimal $\mu_t$ and $\alpha_t$, to get
        \begin{align*}
            -v_t(t,\mu_t)
            &=\mathcal \cH\left(t, \mu_t, \alpha_t, v_\mu(t,\mu_t), v_{y\mu}(t,\mu_t), v_{\mu\mu}(t,\mu_t)\right)\\
            &=\min_{u\in U}\mathcal \cH\left(t, \mu_t, u, v_\mu(t,\mu_t), v_{y\mu}(t,\mu_t), v_{\mu\mu}(t,\mu_t)\right),
        \end{align*}
        for a.e. $t\in[0,T]$, and $\P^0$-a.s. as $\mu_t$ is only random with respect to $\Omega^0$. This supports our maximum principle, which is also only almost sure with respect to $\Omega^0$.\\
        Further, we now expect
        \begin{equation*}
            p_t
            =v_\mu(t,\mu_t)(X_t),\quad
            P_t
            =v_{y\mu}(t,\mu_t)(X_t),\quad\text{and}\quad 
            \mathfrak{P}_t
            =v_{\mu\mu}(t,\mu_t)(X_t,\hat{X}_t),
        \end{equation*}
        where $\hat{X}$ denotes a conditionally independent copy of $X$ given the
        common noise. Thus the second derivative of the lifted value function is not
        represented by a single matrix-valued process. It decomposes into a local part,
        corresponding to $v_{y\mu}$, and a non-local part, corresponding
        to $ v_{\mu\mu}$. The latter is precisely the object represented in our
        maximum principle by the second-level second-order adjoint process $\mathfrak{P}$.\\
        A heuristic argument shows that this is the right identification. We take Lions derivatives of
        $\widehat g$, which is the terminal condition of the Bellman equation above. The first one is given by
        \begin{equation*}
            \widehat g_\mu(\mu)(y)
            =g_x(y,\mu) 
            +\int_{\mathbb R^d}g_\mu(x,\mu)(y)\,\mu(\mathrm{d}x).
        \end{equation*}
        Consequently, along the optimal conditional law
        $\mu_T=\mathcal L(X_T\mid\mathcal F_T^0)$, one obtains
        \begin{equation*}
            v_\mu(T,\mu_T)(X_T)
            =g_x(X_T,\mu_T)
            +\tilde{\E}\left[g_\mu(\tilde{X}_T,\mu_T)(X_T)
            \right],
        \end{equation*}
        which is precisely the terminal condition of the first adjoint process $p$.\\
        Differentiating once more with respect to the lifted variable $y$, we obtain
        \begin{equation*}
            \widehat g_{y\mu}(\mu)(y)
            =g_{xx}(y,\mu)
            +\int_{\mathbb R^d} g_{y\mu}(x,\mu)(y)\, \mu(\mathrm{d}x).
        \end{equation*}
        Hence,
        \begin{equation*}
            v_{y\mu}(T,\mu_T)(X_T)
            =g_{xx}(X_T,\mu_T)
            +\tilde{\E}\left[g_{y\mu}(\tilde{X}_T,\mu_T)(X_T)\right],
        \end{equation*}
        which coincides with the terminal condition of the first-level second-order
        adjoint process $P$.\\
        Finally, the two-times Lions derivative of $\widehat g$ is
        \begin{align*}
            \widehat g_{\mu\mu}(\mu)(y,\hat{y})
            =&g_{\mu x}(y,\mu)(\hat{y})
            +\int_{\mathbb R^d} g_{\mu\mu}(x,\mu)(y,\hat y)\,\mu(\mathrm{d}x)
            + g_{x\mu}(\hat{y},\mu)(y)\\
            =& \int_{\mathbb R^d} g_{\mu\mu}(x,\mu)(y,\hat y)\,\mu(\mathrm{d}x) 
            +g_{x\mu}(\hat y,\mu)(y)
            +g_{x\mu}(y,\mu)(\hat y)^*.
        \end{align*}
        Therefore, evaluating at $(y,\hat y)=(X_T,\hat{X}_T)$, we get
        \begin{align*}
            v_{\mu\mu}(T,\mu_T)(X_T,\hat{X}_T)
            &=  \tilde{\E}\left[
                g_{\mu\mu}(\tilde{X}_T,\mu_T)
                (X_T,\hat{X}_T)\right]
            +g_{x\mu}(\hat{X}_T,\mu_T)(X_T)
            +g_{x\mu}(X_T,\mu_T)(\hat{X}_T)^* .
        \end{align*}
        This is exactly the terminal condition of the second-level second-order adjoint process $\mathbb P$.
    \end{remark}\noindent
    We leave a rigorous treatment of these connections, especially in the non-smooth setting, to later works.

\section{A Simple Example}\label{sec:An easy Example}
    To make the derivation of Peng's maximum principle more clear, we look at an easy example, focusing on the key points.
    Take 
    \begin{equation} \label{eq:state_example}
        \mathrm{d}X_t=A(\alpha_t)\dt+B(\alpha_t)\dW^1_t+C(\alpha_t)\dW^0_t,\quad X_0=x_0\in L^2(\Omega^1,\R^d)
    \end{equation}
    with cost functional 
    \begin{equation}\label{eq:cost_example}
        J(\alpha)
        = \mathbb{E}\left[
            \int_0^T f(\mu_t,\alpha_t)\dt
            + g(\mu_T)
        \right],
    \end{equation}
    where $\mu_t=\cL(X_t\mid \cF^0_t)$. Here, $X$ is not even defined by a true SDE but really just by integrals over control dependent functions.
    Under the same spike variation as introduced in Section \ref{sec:The Variational Equations} we get the first-order variational equation
    \begin{equation}
        \begin{aligned}
            \mathrm{d}Y_t 
            =& \delta A(t) \dt 
            + \delta B(t) \dW^1_t
            + \delta C(t) \dW^0_t,\quad 
            Y_0 = 0,
        \end{aligned}\label{eq:first variational process SDE_example}
    \end{equation}
    which again is not even a true SDE, and the second-order variational equation is even 
    \begin{equation}
        \begin{aligned}
            \mathrm{d}Z_t =0,\quad Z_0=0.
        \end{aligned}\label{eq:second variational process SDE_example}
    \end{equation}
    These are the (Taylor) approximations of $X-X^\epsilon$. In this particular case $\Delta X= Y$ and $Z\equiv 0$ so that the estimations from Lemma \ref{lemma:Estimates on Y and Z} and \ref{lemma:estimate on DeltaX-Y-Z}, i.e. 
    \begin{equation*}
        \mathbb{E}\left[\sup _{t \in[0, T]}\left\|\Delta X_t-Y_t\right\|^{2 k}\right] \in o(\epsilon^{2k}).
    \end{equation*}
    hold trivially.\\
    Now, we can expand our cost functional (according to Lemma \ref{lemma:Taylor formula for Lions Derivative}) to deduce a necessary condition for optimality of an $\alpha$. Even though, $Z\equiv 0$ and thus only a first-order approximation of $X$ is done, we still need to use a second-order Taylor expansion for the cost functional, as we still only get $\mathbb{E}[\sup _{t \in[0, T]}\left\|\Delta X_t\right\|^{2 k}] \in O(\epsilon^k)$, so the remainder of the Taylor expansion is only of order $o(\epsilon)$ for an at least second-order expansion.\\
    Since $Y=\Delta X$, we substitute it and get
    \begin{align}
        0\leq J(\alpha^{\epsilon})-J(\alpha)
        = & \E\Bigg[ \int_0^T \tilde{\mathbb{E}}\left[f_\mu(\theta_t)(\tilde{X}_t)[\tilde{Y}_t]\right]
        +\delta f(t)\nonumber\\
        &+\frac{1}{2}\tilde{\mathbb{E}}\left[f_{y\mu}(\theta_t)(\tilde{X}_t)[\tilde{Y}_t,\tilde{Y}_t]\right]
        +\frac{1}{2}\tilde{\tilde{\E}} \left[\tilde{\mathbb{E}} \left[f_{\mu\mu}(\theta_t)(\tilde{X}_t,\tilde{\tilde{X}}_t)[\tilde{Y}_t,\tilde{\tilde{Y}}_t]\right]\right] \dt\nonumber\\
        & +\tilde{\mathbb{E}}\left[g_\mu(\mu_T)(\tilde{X}_T)[\tilde{Y}_T]\right] \label{eq:First Duality Replacement_example}\\
        & +\frac{1}{2}\tilde{\mathbb{E}}\left[g_{y\mu}(\mu_T)(\tilde{X}_T)[\tilde{Y}_T,\tilde{Y}_T]\right] \label{eq:second Duality Replacement_example}
        \\
        &+\frac{1}{2}\tilde{\tilde{\E}} \left[\tilde{\mathbb{E}} \left[g_{\mu\mu}(\mu_T)(\tilde{X}_T,\tilde{\tilde{X}}_T)[\tilde{Y}_T,\tilde{\tilde{Y}}_T]\right]\right]\Bigg]+o(\epsilon), \label{eq:third Duality Replacement_example}
    \end{align}
    where $\tilde{X},\tilde{Y}, \tilde{\tilde{X}},\tilde{\tilde{Y}}$ denote (conditionally on $\cF^0$) independent copies (in the sense of Section \ref{sec:The Variational Equations}).\\
    This already constitutes a necessary condition for optimality of $\alpha$, but the condition is not explicit, as $Y$ implicitly depends on the spike variation $\alpha^\epsilon$. To check this condition for many different spike variations, one would have to solve the equation for $Y$ every single time. Further, we would like to have a local or infinitesimal condition, i.e. for fixed $t\in[0,T]$. Thus, we aim to replace all the terms containing $Y$ by so-called dualization.
    This is where the adjoint processes come into play. Their terminal values need to be chosen in such a way that they appear in the above expansion in a dualization with $Y$ or a quadratic term of $Y$ and at the same time their time evolution needs to be chosen in such a way that this replacement also cancels the unwanted terms, hence requiring backward SDEs.\\
    The problem in the case of distribution dependence is that some of the above terms contain (conditionally) independent copies of $Y$, which makes a dualization more difficult. Here, the second-level second-order adjoint equation comes into play, as it is defined on a bigger probability space (as defined in Section \ref{subsec:Lifting and Third Adjoint}) that allows for these copies of $Y$.
    \begin{remark} \label{remark:example_ConditionalExpNotOfHigherOrder}
        Notice that due to the conditional law we cannot argue in a similar fashion to \cite{Buckdahn2016,guatteri2025pengsmaximumprinciplestochastic,chen2026transpositionapproachoptimalcontrol} that the terms containing two (conditionally) independent copies of $Y$ are of lower order, as even for deterministic controls
        \begin{equation*}
            \E[Y_T\mid \cF^0_T]\sim \int_{E_\epsilon} \delta C(t)\dW^0_s\notin o(\epsilon^{\frac{1}{2}}).
        \end{equation*}
    \end{remark}\noindent
    Towards the above mentioned dualization, we define the first-order adjoint process by
    \begin{equation}
        \begin{aligned}
            \mathrm{d}p_t 
            =& - \tilde{\E}\left[f^*_\mu(\tilde{\theta}_t)(X_t)\right] \dt  + q_t \dW^1_t+ r_t \dW^0_t, \\
            p_T &= \tilde{\E}[g_\mu(\mu_T)(X_T)],
        \end{aligned}\label{eq:First order adjoint_example}
    \end{equation}
    the first-level second-order adjoint process by
    \begin{equation}
        \begin{aligned}
            \mathrm{d}P_t =& -\tilde{\E}[ f_{y\mu}(\tilde{\theta}_t)(X_t) ]\dt
            + Q_t \dW^1_t
            + R_t \dW^0_t, \\
            P_T =& \tilde{\E}[g_{y\mu}(\mu_T)( X_T)]
        \end{aligned}
        \label{eq:First-Level second-order adjoint equation_example}
    \end{equation}
    and the second-level second-order adjoint process by
    \begin{equation}
        \begin{aligned}
            \mathrm{d}\mathfrak{P}_t
            =&-\tilde{\E} \left[f_{\mu\mu}(\tilde{\theta}_t)(X_t,\hat{X}_t)\right]\dt
            + \mathfrak{Q}^{(1)}_t \dW^1_t
            +\mathfrak{Q}^{(2)}_t \,\mathrm{d}\hat{W}^1_t
            +\mathfrak{R}_t \dW^0_t,\\
            \mathfrak{P}_T=&\tilde{\E}[g_{\mu \mu}(\mu_T)( X_T,\hat{X}_T)],
        \end{aligned}\label{eq:Second-Level Second-Order Adjoint_example}
    \end{equation}
    where again $\hat{X}$ denotes a (conditionally on $\cF^0$) independent copy of $X$ (in the sense of Section \ref{subsec:Lifting and Third Adjoint}).\\
    By Itô's formula, this results in 
    \begin{equation}
        \begin{aligned}
            &\bar{\E}\left[\tilde{\mathbb{E}}\left[g_\mu(\mu_T)(\tilde{X}_T)[\tilde{Y}_T]\right]\right]\\
            =&\bar{\E}\left[\langle p_T, Y_T\rangle\right]\\
            =& \bar{\E}\bigg[\int_0^T - \tilde{\E}\left[f_\mu(\theta_t)(\tilde{X}_t)[ \tilde{Y}_t]\right]
            +\langle p_t,\delta A(t) \rangle 
            + \langle q_t,\delta B(t) \rangle
            + \langle r_t,\delta C(t) \rangle\dt\bigg],
        \end{aligned} \label{eq:duality p and Y_example}
    \end{equation}
    in 
    \begin{equation}
        \begin{aligned}
            &\bar{\E}\left[\tilde{\mathbb{E}}\left[g_{y\mu}(\mu_T)(\tilde{X}_T)[\tilde{Y}_T,\tilde{Y}_T]\right]\right]\\
            =&\bar{\E}\left[\left\langle P_T, Y_T \otimes Y_T\right\rangle\right]\\
            =&\bar{\E}\bigg[\int_0^T-\tilde{\E}[ f_{y\mu}(\tilde{\theta}_t)(X_t) ][Y_t ,Y_t]
            +\langle P_t, \delta B(t)\delta B^*(t)
            +\delta C(t)\delta C^*(t)\rangle \dt\bigg]+o(\epsilon),
        \end{aligned}\label{eq:first second duality_example}
    \end{equation}
    and, using Fubini, in 
    \begin{equation}
        \begin{aligned}
            &\bar{\E}\left[\tilde{\tilde{\E}} \left[\tilde{\mathbb{E}} \left[g_{\mu\mu}(\mu_T)(\tilde{X}_T,\tilde{\tilde{X}}_T)[\tilde{Y}_T,\tilde{\tilde{Y}}_T]\right]\right]\right]\\
            =&\bar{\E}\left[\left\langle \mathfrak{P}_T, Y_T \otimes \hat{Y}_T\right\rangle\right]\\
            =&\bar{\E}\bigg[\int_0^T-\tilde{\E} \left[f_{\mu\mu}(\tilde{\theta}_t)(X_t,\hat{X}_t)[ Y_t,\hat{Y}_t]\right]+\langle \mathfrak{P}_t ,\delta C(t)\delta \hat{C}^*(t)\rangle \dt\bigg] +o(\epsilon).
        \end{aligned}\label{eq:second second duality_example}
    \end{equation}
    Thus, we replace \eqref{eq:First Duality Replacement_example}, \eqref{eq:second Duality Replacement_example} and \eqref{eq:third Duality Replacement_example} by \eqref{eq:duality p and Y_example}, \eqref{eq:first second duality_example} and \eqref{eq:second second duality_example}, giving
    \begin{align*}
        0\leq & J(\alpha^{\epsilon})-J(\alpha )\\
        = & \bar{\E} \left[ \int_0^T \delta H(t)
        +\frac{1}{2}\big\langle P_t,
        \delta B(t)\delta B(t)^*
        +\delta C(t)\delta C^*(t)\big\rangle
        +\langle \mathfrak{P}_t ,\delta C(t)\delta \hat{C}^*(t)\rangle \dt\right] +o(\epsilon),
    \end{align*}
    i.e. the desired explicit necessary condition, which can be specified to the maximum principle from Theorem \ref{Theorem:maximum principle}.

\section*{Acknowledgments}
WS and JBS acknowledge support from DFG CRC/TRR 388 'Rough Analysis, Stochastic Dynamics and Related Fields', Projects A10 and B09.

\appendix
\section{Taylor's formula for $C^{2,1}$-Functions} \label{appendix:Taylor formula for x and mu}
    A second-order Taylor expansion will be needed. It works under the assumptions given in Section \ref{sec:Preliminaries}. The first expansion, taken from \citet{BuckdahnPeng2017}, is for purely measure-dependent functions.
    \begin{lemma}[\cite{BuckdahnPeng2017} Lemma 2.1] \label{lemma:Taylor formula for Lions Derivative}
        Let $\varphi \in C_b^{2,1}(\mathcal{P}_2(\mathbb{R}^d),\R)$. Then, for any given $X \in L^2(\Omega, \mathbb{R}^d)$ we have the following second-order expansion, for all $Y \in L^2(\Omega, \mathbb{R}^d)$,
        \begin{equation*}
            \begin{aligned}
                \varphi(\cL(Y))- & \varphi(\cL(X)) \\
                = & \E\left[\varphi_{\mu}(\cL(X))( X) [\eta]\right]
                +\frac{1}{2} \E\left[\tilde{\E}\left[\varphi_{\mu\mu}(\cL(X))(\tilde{X}, X)[\tilde{\eta},\eta]\right]\right] \\
                & +\frac{1}{2} \E\left[\varphi_{y\mu}(\cL(X))( X) [\eta , \eta]\right]+R(\cL(Y), \cL(X)),
            \end{aligned}
        \end{equation*}
        where $\eta:=Y-X$. Further, for all $Y \in L^2(\Omega, \mathbb{R}^d)$, the remainder satisfies the estimate
        \begin{equation*}
            \begin{aligned}
                |R(\cL(Y), \cL(X))| & \leq C((\E[|Y-X|^2])^{3 / 2}+\E[|Y-X|^3]) \\
                & \leq C \E[|Y-X|^3],
            \end{aligned}
        \end{equation*}
        and the constant $C \in \mathbb{R}_{+}$ only depends on the Lipschitz constants of $\varphi_{\mu\mu}$ and $\varphi_{y\mu}$.
    \end{lemma}
    We generalize the above Taylor expansion to the setting where the function might also depend on another variable in $\R^d$.
    \begin{lemma} \label{lemma:Taylor formula for x and mu dependence}
        Let $\varphi \in C_b^{2,1}(\R^d\times \mathcal{P}_2(\mathbb{R}^d),\R)$. Then, for any given $x\in\R^d$ and $X \in L^2(\Omega, \mathbb{R}^d)$ we have the following second-order expansion, for all $y\in\R^d$ and $Y\in L^2(\Omega,\R^d)$,
        \begin{equation*}
            \begin{aligned}
                &\varphi(y,\cL(Y))- \varphi(x,\cL(X)) \\
                = & \varphi_x(x,\cL(X))[y-x]
                +\E\left[\varphi_\mu(x,\cL(X))(X)[Y-X]\right]
                +\frac{1}{2}\varphi_{xx}(x,\cL(X))[y-x,y-x]\\
                &+\frac{1}{2} \E\left[\tilde{\E}\left[ \varphi_{\mu\mu}(x,\cL(X))(\tilde{X}, X)[ \tilde{Y}-\tilde{X},Y-X]\right]
                +\varphi_{y\mu}(x,\cL(X))(X)[Y-X,Y-X]\right]\\
                &+\E\left[\varphi_{x\mu}(x,\cL(X))(X)[Y-X,y-x]\right]+R(x,y,\cL(X), \cL(Y)).
            \end{aligned}
        \end{equation*}
        Furthermore, for all $x,y\in \R^d$ and $X,Y \in L^2(\Omega, \mathbb{R}^d)$, the remainder satisfies the estimate
        \begin{equation*}
            \begin{aligned}
                |R(x,y,\cL(X), \cL(Y))| \leq C\left( |y-x|^3
                + \E[|Y-X|^3]\right),
            \end{aligned}
        \end{equation*}
        and the constant $C \in \mathbb{R}_{+}$ only depends on the Lipschitz constants of $\varphi_{xx}$, $\varphi_{x\mu}$, $\varphi_{\mu\mu}$ and $\varphi_{y\mu}$.
    \end{lemma}
    \begin{proof}
        Denote $\mu=\cL(X)$ and $\nu=\cL(Y)$ and decompose
        \begin{align*}
            \varphi(y, \nu)- \varphi(x,\mu)
            =\underbrace{\varphi(y,\nu)-\varphi(x,\nu)}_{I_1}+\underbrace{\varphi(x,\nu)-\varphi(x,\mu)}_{I_2}.
        \end{align*}
        Clearly, $I_2$ can be expanded by the above Lemma \ref{lemma:Taylor formula for Lions Derivative}, so we only have to treat $I_1$. For fixed $\nu$, $\varphi(\cdot,\nu)\in C^2(\R^d,\R)$. Thus, by the standard Taylor formula,
        \begin{align*}
            I_1=&\varphi_x(x,\nu)[y-x]+\frac{1}{2}\varphi_{xx}(x,\nu)[y-x,y-x]\\
            &+\int_0^1 (1-\lambda)\left(\varphi_{xx}(x+\lambda(y-x),\nu)-\varphi_{xx}(x,\nu)\right)[y-x,y-x]\,\mathrm{d}\lambda.
        \end{align*}
        Further, $\varphi_x(x,\cdot)$ is continuously Lions differentiable, so by the standard Taylor formula (cf. \cite{CarmonaDelarue2018I} (5.33) and (5.34)), using an explicit form of the remainder,
        \begin{align*}
            \varphi_x(x,\nu)[y-x]
            =&\left(\varphi_x(x,\nu)-\varphi_x(x,\mu)\right)[y-x]+\varphi_x(x,\mu)[y-x]\\
            =& \E\left[\varphi_{x\mu}(x,\mu)(X)[Y-X]\right][y-x]
            +\varphi_x(x,\mu)[y-x]\\
            &+\int_0^1(1-\lambda)\E\Big[\big(\varphi_{x\mu}(x,\cL(X+\lambda(Y-X))(X+\lambda(Y-X))\\
            &\qquad\qquad\qquad\quad-\varphi_{x\mu}(x,\mu)(X)\big)[Y-X]\Big]\,\mathrm{d}\lambda[y-x].
        \end{align*}
        And lastly, 
        \begin{align*}
            \frac{1}{2}\varphi_{xx}(x,\nu)[y-x,y-x]
            =&\frac{1}{2}\left(\varphi_{xx}(x,\nu)
            -\varphi_{xx}(x,\mu)\right)[y-x,y-x]\\
            &+\frac{1}{2}\varphi_{xx}(x,\mu)[y-x,y-x].
        \end{align*}
        Putting it all together, we get the desired formula with remainder
        \begin{align*}
            R(x,y,\mu,\nu)
            =&\tilde{R}(\mu,\nu)
            +\frac{1}{2}\left(\varphi_{xx}(x,\nu)
            -\varphi_{xx}(x,\mu)\right)[y-x,y-x]\\
            &+\int_0^1 (1-\lambda)\left(\varphi_{xx}(x+\lambda(y-x),\nu)-\varphi_{xx}(x,\nu)\right)[y-x,y-x] \\
            &\qquad+ (1-\lambda)\E\Big[\varphi_{x\mu}(x,\cL(X+\lambda(Y-X))(X+\lambda(Y-X))[Y-X,y-x]\\
            &\qquad\qquad\qquad\quad-\varphi_{x\mu}(x,\mu)(X)[Y-X,y-x]\Big]\,\mathrm{d}\lambda,
        \end{align*}
        where $\tilde{R}$ is the remainder from the above Lemma \ref{lemma:Taylor formula for Lions Derivative}. The estimation for the remainder is clear and the constants only depend on Lipschitz constants of the second-order derivatives.
    \end{proof}

\section{Existence and Uniqueness for McKean--Vlasov Equations with Common Noise} \label{appendix:Ex and Uni for Controlled MKVSDE with Common Noise}
    In this section, we will discuss the well-posedness of forward and backward SDEs given in the current paper. Under this consideration, we will work in the special setup given in Section \ref{subsec:Lifting and Third Adjoint}, which fits our purpose well and simplifies notations. We stress that these results hold more generally.\\
    Recall the setting from Section \ref{subsec:Lifting and Third Adjoint}. We take as a stochastic basis $(\Omega,\P):=\bigotimes_{i=0}^2 (\Omega^i,\mathbb{P}^i)$, where, for $i\in\{0,1\}$, $(\Omega^i,\cF^i,\mathbb{P}^i)$ is a complete filtered probability space such that $\cF^0=(\mathcal{F}^0_t)_{0\le t\le T}$ is the augmented filtration generated by a $d$-dimensional Brownian motion $W^0$, $\cF^1=(\mathcal{F}^1_t)_{0\le t\le T}$ is the augmented filtration generated by a $d$-dimensional Brownian motion $W^1$ and some independent random variable $x_0\in L^2(\Omega^1,\R^d)$ and $(\Omega^2,\cF^2,\mathbb{P}^2,W^2):=(\Omega^1,\cF^1,\mathbb{P}^1,W^1)$. On this space we consider the completed version $\cF$ of the product filtration $\bigotimes_{i=0}^2 \cF^i$. By an abuse of notation we also consider $\cF^0_t:= \cF^0_t\times\{\emptyset,\Omega^1\}\times\{\emptyset,\Omega^2\}$ as a sub-$\sigma$-algebra of $\cF_t$ and similarly for $\cF^1$ and $\cF^2$ and also consider all random variables on the product space even if they are a priori only defined on one of the components.\\
    For our notion of solutions, we define the space $\mathbb{S}^{2, m}$ of all (equivalence classes of) $(\cF_t)_{t\in[0,T]}$-progressively measurable continuous processes $\Phi:\Omega\times[0,T]\to \R^m$ satisfying $\mathbb{E}[\sup _{t\in[0,T]}|\Phi_t|^2]<\infty$, equipped with the norm $\|\Phi\|_{\mathbb{S}}^2=\mathbb{E}[\sup _{0 \leq t \leq T}|\Psi_t|^2]$ and
    the space $\Lambda^{2, m}$ of all (equivalence classes of) $(\cF_t)_{t\in[0,T]}$-progressively measurable processes $\Psi:\Omega\times[0,T]\to \R^m$ satisfying $\mathbb{E}[\int_0^T|\Psi_t|^2dt]<\infty $, equipped with the norm $\|\Psi\|_{\Lambda}^2=\mathbb{E}[\int_0^T|\Psi_t|^2dt]$.
    \subsection{Conditional McKean--Vlasov Forward Stochastic Differential Equations}
        We now first solve a conditional McKean--Vlasov SDE
        \begin{equation}\label{eq:state in appendix}
            \mathrm{d}X_t 
            = b(t,X_t,\mu_t)\dt 
            + \sigma(t,X_t,\mu_t)\dW^1_t
            +\sigma^0(t,X_t,\mu_t)\dW^0_t,
            \qquad X_0 = x_0,
        \end{equation}
        which lives solely on $\Omega^0\times \Omega^1$ and where $\mu_t:=\cL(X_t\mid \cF^0_t)$ and
        \begin{align*}
            b:&\quad\Omega^0\times \Omega^1\times[0,T]\times \R^m \times \cP_2(\R^m) \to \R^m\\
            \sigma:&\quad \Omega^0\times \Omega^1\times[0,T]\times \R^m \times \cP_2(\R^m) \to \R^{m\times d}\\
            \sigma^0:& \quad\Omega^0\times \Omega^1\times[0,T]\times \R^m \times \cP_2(\R^m) \to \R^{m\times d}
        \end{align*}
        are progressively measurable with respect to the completion of $\cF^0\otimes \cF^1$.
        \begin{definition}\label{Def:Solution to Cond MKV-SDE}
            On the probabilistic set-up $(\Omega,\cF,(\cF_t)_{t\in[0,T]}, \P)$ we call a strong solution to the conditional McKean--Vlasov SDE \eqref{eq:state in appendix} on the interval $[0, T]$ a process $X\in \mathbb{S}^{2,m}$ such that $\P$-a.s. for every $t\in[0,T]$,
            \begin{equation*}
                X_t 
                =x_0
                +\int_0^t b(s,X_s,\mu_s)\ds 
                + \int_0^t \sigma(s,X_s,\mu_s)\dW^1_s
                + \int_0^t\sigma^0(s,X_s,\mu_s)\dW^0_s,
            \end{equation*}
            where $\mu_s=\cL(X_s\mid \cF^0_s)$, such that the integrals are all well-defined.
        \end{definition}
        Note the discussion of measurability in Remark \ref{remark:after def of variational eq} such that all integrals are well-defined.
        \begin{assumption}\label{Assumption:Solution to Cond MKV-SDE}
            Let $b,\sigma,\sigma^0$ be Lipschitz and bounded with constants uniform in $t$ and $\omega$.
        \end{assumption}
        \begin{theorem} \label{theorem_appendix:ex and uni of MKV SDE w common noise}
            Under Assumption \ref{Assumption:Solution to Cond MKV-SDE}, there exists a unique strong solution to \eqref{eq:state in appendix} in the sense of Definition \ref{Def:Solution to Cond MKV-SDE}.
        \end{theorem}
        \begin{proof}
            A proof can be found in \cite{Djete2022} Theorem A.3.
        \end{proof}
        
        \begin{remark} \label{remark_appendix:Conditional Copy of SDE}
            Note, that this equation can be solved on $\Omega^0\times \Omega^1$ alone and we do not need to address $\Omega^2$ in any way, but the equation can easily be lifted to the whole space by defining $X_t(\omega)=X_t(\omega_0,\omega_1,\omega_2):=X_t(\omega_0,\omega_1)$.\\
            Now, consider
            \begin{equation}\label{eq:state in appendix lifted}
                \mathrm{d}\hat{X}_t 
                = \hat{b}(t,\hat{X}_t,\hat{\mu}_t)\dt 
                + \hat{\sigma}(t,\hat{X}_t,\hat{\mu}_t)\dW^2_t
                +\hat{\sigma}^0(t,\hat{X}_t,\hat{\mu}_t)\dW^0_t,
                \qquad \hat{X}_0 = \hat{x}_0,
            \end{equation}
            where $x_0\in L^2(\Omega^1,\R^d)$ and $\hat{\mu}_t:=\cL(\hat{X}_t\mid \cF^0_t)$ and $\hat{x}_0,\hat{b},\hat{\sigma},\hat{\sigma}^0$ are equal to $x_0,b,\sigma,\sigma^0$ but instead of $\Omega^1$ they now depend on $\Omega^2$.\\
            If \eqref{eq:state in appendix} has a unique strong solution $X$, then $\hat{X}_t(\omega):=X(\omega_0,\omega_2)$ is a strong solution to \eqref{eq:state in appendix lifted} and this solution is unique up to indistinguishability.
        \end{remark}
    \subsection{Conditional McKean--Vlasov Backward Stochastic Differential Equations}
        Now, we turn towards the conditional McKean--Vlasov backward SDE. For this let $\chi_t:\Omega\to E$ be a progressively measurable process with values in some metric space $(E,e)$ equipped with its Borel $\sigma$-algebra.
        To briefly write conditional expectations with respect to $\cF^0\times \cF^1\times \{\emptyset,\Omega^2\}$ and $\cF^0\times \{\emptyset,\Omega^1\}\times \cF^2$, we denote $(\tilde{\Omega},\tilde{\cF},\tilde{\P})=(\Omega^1,\mathcal{F}^1,\mathbb{P}^1)\otimes (\Omega^2,\mathcal{F}^2,\mathbb{P}^2)$ and define for any random variable $X:\Omega\to \R^d$ its conditional lift with respect to $\Omega^1$
    \begin{equation*}
        \tilde{X}^1:
        \Omega \times \tilde{\Omega}\to \R^d,\quad 
        (\omega,\tilde{\omega})
        =(\omega_0,\omega_1,\omega_2,\tilde{\omega}_1,\tilde{\omega}_2)\mapsto X(\omega_0,\tilde{\omega}_1,\omega_2)
    \end{equation*}
    and also its conditional lift with respect to $\Omega^2$
    \begin{equation*}
        \tilde{X}^2:
        \Omega \times \tilde{\Omega}\to \R^d,\quad 
        (\omega,\tilde{\omega})
        =(\omega_0,\omega_1,\omega_2,\tilde{\omega}_1,\tilde{\omega}_2)\mapsto X(\omega_0,\omega_1,\tilde{\omega}_2)
    \end{equation*}
    and also its 'full' lift given with respect to $\Omega^1\times\Omega^2$
    \begin{equation*}
        \tilde{X}^{1,2}:
        \Omega \times \tilde{\Omega}\to \R^d,\quad 
        (\omega,\tilde{\omega})
        =(\omega_0,\omega_1,\omega_2,\tilde{\omega}_1,\tilde{\omega}_2)\mapsto X(\omega_0,\tilde{\omega}_1,\tilde{\omega}_2)
    \end{equation*}
    Then, for a measurable, bounded $\varphi$, and random variables $X,Y:\Omega\to \R^d$
    \begin{equation*}
        \tilde{\E}\left[\varphi(Y,\tilde{X}^1)\right]:\Omega \to \R^d ,\quad \omega\mapsto \tilde{\E}\left[\varphi(Y(\omega),X(\omega_0,\tilde{\omega}_1,\omega_2))\right]:= \int_{\Omega^1}\varphi(Y(\omega),X(\omega_0,\tilde{\omega}_1,\omega_2))\,\P^1(\mathrm{d}\tilde{\omega})
    \end{equation*}
    constitutes a measurable (due to Fubini) random variable, which can be seen as dependent on the conditional law of $X$ -- in this case with respect to $\cF^0\times \{\emptyset,\Omega^1\}\times \cF^2$ (cf. Lemma 2.4 \cite{CarmonaDelarue2018II}) -- while keeping in $Y$ the original random variable. This works similarly for the other lifts.\\
        For ease of notation, denote,
        \begin{equation*}
            \tilde{\Phi}:=(\tilde{\Phi}^1,\tilde{\Phi}^2,\tilde{\Phi}^{1,2}),
        \end{equation*}
        where $\Phi:\Omega\to D$ and $(D,d)$ is some metric space. This makes it possible to denote linear dependence on the conditional laws of the random variables with respect to $\cF^0\times \{\emptyset,\Omega^1\}\times \cF^2$, $\cF^0\times \cF^1\times \{\emptyset,\Omega^2\}$ and $\cF^0\times \{\emptyset,\Omega^1\}\times \{\emptyset,\Omega^2\}$.\\
        We consider the equation
        \begin{equation}
            \begin{aligned}
                \mathrm{d}P_t=&\tilde{\E}\left[F\left(t,
                P_t,\tilde{P}_t,
                Q_t,\tilde{Q}_t,
                S_t,\tilde{S}_t,
                R_t,\tilde{R}_t,
                \chi_t,\tilde{\chi}_t\right)\right]\dt +R_t\dW^0_t+Q_t\dW^1_t+S_t\dW^2_t\\
                P_T=&\xi \in L^2(\Omega,\cF_T;\R^{m}),
            \end{aligned}\label{eq:Cond-MKVBSDE}
        \end{equation}
        where 
        \begin{equation*}
            F: [0,T]\times (\R^m)^4\times (\R^{m\times d})^{12}\times E^4 \to \R^m
        \end{equation*}
        is product measurable.
        \begin{remark}
            \begin{enumerate}[(i)]
                \item This equation is of conditional McKean--Vlasov type (cf. \cite{CarmonaDelarue2018II} Lemma 2.4).
                \item Note the dependence of the drift on the joint conditional laws of $(P,Q,R,S,\chi)$. This is stronger than just taking coefficients that depend on the conditional law of $(P,Q,R,S)$. This also makes clear why we do not let $F$ be random instead of introducing $\chi$, as then we would have to introduce lifts of $F$ which would make the notation even heavier.
                \item The simple linear dependence on the joint conditional law makes it possible to write the drift using the $\tilde{\E}$-expectation. Still, this formulation is well-suited for treating our adjoint equations as these are always linear in the conditional law, so we do not have to treat general conditional McKean--Vlasov backward SDE, where the coefficients might depend more generally on the joint conditional law of the processes. We do emphasize that this would not be troublesome but we do not make this generalization for brevity- and notation-sake.
                \item In the application, $\chi$ takes the role of $(X,\mu,\alpha)$, i.e. the state, the conditional law of the state and the control process. One might also imagine further possible randomness coming directly from random coefficients etc.
            \end{enumerate}
        \end{remark}
        \begin{definition} \label{definition:solution Cond-MKVBSDE}
            On the probabilistic set-up $(\Omega,\cF,(\cF_t)_{t\in[0,T]}, \P)$ we call a strong solution to the conditional McKean--Vlasov BSDE \eqref{eq:Cond-MKVBSDE} on the interval $[0, T]$ a $4$-tuple $(P,Q,R,S)\in \mathbb{S}^{2,m}\times\Lambda^{2,m\times d}\times\Lambda^{2,m\times d}\times\Lambda^{2,m\times d}$ such that $\P$-a.s. for every $t\in[0,T]$,
            \begin{equation*}
                \begin{aligned}
                   P_t=&\xi-\int_t^T\tilde{\E}\left[F\left(t,
                    P_t,\tilde{P}_t,
                    Q_t,\tilde{Q}_t,
                    S_t,\tilde{S}_t,
                    R_t,\tilde{R}_t,
                    \chi_t,\tilde{\chi}_t\right)\right]\dt \\
                    &-\int_t^T R_t\dW^0_t-\int_t^TQ_t\dW^1_t-\int_t^TS_t\dW^2_t,
                \end{aligned}
            \end{equation*}
            such that all integrals are well-defined.
        \end{definition}
        \begin{assumption} \label{Assumption:existence and uniqueness of Cond-MKVBSDE}
            For fixed $e_1,e_2,e_3,e_4\in E$ let 
            \begin{equation*}
                F(t,\cdot,e_1,e_2,e_3,e_4):(\R^m)^4\times (\R^{m\times d})^{12}\to \R^m
            \end{equation*}
            be Lipschitz and bounded with constant $K$ independent of the choice of $e_1,e_2,e_3,e_4$.
        \end{assumption}
        \begin{theorem} \label{theorem_appendix:existence and uniqueness of Cond-MKVBSDE}
            Under Assumption \ref{Assumption:existence and uniqueness of Cond-MKVBSDE}, a solution to \eqref{eq:Cond-MKVBSDE} in the sense of Definition \ref{definition:solution Cond-MKVBSDE} exists and is unique.
        \end{theorem}
        \begin{proof}
            We will make a classical fixed point argument on the space
            \begin{equation*}
                \mathbb{L}:=
                C\left([0, T], L^2\left(\Omega,\R^{m}\right)\right)
                \times L^2([0,T]\times\Omega,\R^{m\times d})
                \times L^2([0,T]\times\Omega,\R^{m\times d})
                \times L^2([0,T]\times\Omega,\R^{m\times d}).
            \end{equation*}
            Now, fix 
            \begin{equation*}
                \Phi:=(p,q,r,s)\in \mathbb{L}
            \end{equation*}
            and consider the standard BSDE
             \begin{equation}
            \begin{aligned}
                \mathrm{d}P_t=&\tilde{\E}\left[F\left(t,
                P_t,\tilde{p}_t,
                Q_t,\tilde{q}_t,
                S_t,\tilde{s}_t,
                R_t,\tilde{r}_t,
                \chi_t,\tilde{\chi}_t\right)\right]\dt
                +R_t\dW^0_t+Q_t\dW^1_t+S_t\dW^2_t\\
                P_T=&\xi \in L^2(\Omega,\cF_T;\R^{m}).
            \end{aligned}\label{eq:Cond-MKVBSDE with fixed in proof}
        \end{equation}
            By Assumption \ref{Assumption:existence and uniqueness of Cond-MKVBSDE}, this BSDE has a unique solution, which we will denote $(P^{\Phi},Q^{\Phi},R^{\Phi},S^{\Phi})\in \mathbb{S}^{2,m}\times\Lambda^{2,m\times d}\times\Lambda^{2,m\times d}\times \Lambda^{2,m\times d}\subset \mathbb{L}$ (cf. \cite{Pardoux2014} Chapter 5,\cite{CarmonaDelarue2018I} Chapter 4). Notice that  $(\cF_t)_{t\geq 0 }$ coincides with the completed filtration generated by $(x_0,W^0_t,W^1_t,W^2_t)_{t\geq 0}$. Now, consider the map
            \begin{equation*}
                \Xi: \mathbb{L}\to \mathbb{L},\quad 
                \Phi\mapsto 
                (P^{\Phi},Q^{\Phi},R^{\Phi},S^{\Phi}).
            \end{equation*}
            Denoting further $\phi:=(\mathfrak{p},\mathfrak{q},\mathfrak{r},\mathfrak{s}) \in \mathbb{L}$, we will show, that $\Xi$ is a contraction, with respect to the metric 
            \begin{align*}
                &\rho(\Phi,\phi)\\
                &:=\sup _{t \in[0, T]} e^{\kappa t} \E\left[\left\|p_t- \mathfrak{p}_t\right\|^2\right]
                +\frac{3}{4}\E\left[\int_0^T e^{\kappa t} \left\| q_t-\mathfrak{q}_t\right\|^2+e^{\kappa t} \left\| r_t-\mathfrak{r}_t\right\|^2+e^{\kappa t} \left\| s_t-\mathfrak{s}_t\right\|^2\dt\right],
            \end{align*}
            where $\kappa>0$ will be chosen later, which gives us a unique solution to our desired equation by usual fixed point theorems. \\
            So denote 
            $\Delta^P=P^{\Phi}-P^{\phi}$, 
            $\Delta^{Q}=Q^{\Phi}-Q^{\phi}$, $\Delta^{R}=R^{\Phi}-R^{\phi}$ and 
            $\Delta^{S}=S^{\Phi}-S^{\phi}$ and $\Delta^p=p-\mathfrak{p},\Delta^{q}=q-\mathfrak{q},\Delta^{r}=r-\mathfrak{r}$ and $\Delta^{s}=s-\mathfrak{s}$. By Ito's formula, for every $\tau\in[0,T]$
            \begin{align*}
                0=&\E[e^{\kappa T}\|\Delta^P_T\|^2]\\
                =&\E\Big[e^{\kappa \tau}\|\Delta^P_\tau\|^2
                -\int_\tau^T e^{\kappa t}2\langle \Delta^P_t,\tilde{\E}\left[F\left(t,
                P^\Phi_t,\tilde{p}_t,
                Q^\Phi_t,\tilde{q}_t,
                S^\Phi_t,\tilde{s}_t,
                R^\Phi_t,\tilde{r}_t,
                \chi_t,\tilde{\chi}_t\right)\right]\\
                &-\tilde{\E}\left[F\left(t,
                P^\phi_t,\tilde{\mathfrak{p}}_t
                Q^\phi_t,\tilde{\mathfrak{q}}_t,
                S^\phi_t,\tilde{\mathfrak{s}}_t,
                R^\phi_t,\tilde{\mathfrak{r}}_t,
                \chi_t,\tilde{\chi}_t\right)\right]\rangle\\
                &+e^{\kappa t}\|\Delta^{Q}_t\|^2 
                +e^{\kappa t}\|\Delta^{R}_t\|^2
                + e^{\kappa t}\|\Delta^{S}_t\|^2 
                +\kappa e^{\kappa t}\|\Delta^{P}_t\|^2\dt \Big]
            \end{align*}
            so rearranging gives
            \begin{align*}
                &\E\Big[ e^{\kappa \tau}\|\Delta^P_\tau\|^2
                +\int_\tau^T e^{\kappa t}\|\Delta^{Q}_t\|^2 
                + e^{\kappa t}\|\Delta^{R}_t\|^2
                + e^{\kappa t}\|\Delta^{S}_t\|^2 \dt 
                +\kappa\int_\tau^T e^{\kappa t}\|\Delta^{P}_t\|^2\dt \Big]\\
                =&\E\Big[\int_\tau^T 2e^{\kappa t}\langle \Delta^P_t,\tilde{\E}\left[F\left(t,
                P^\Phi_t,\tilde{p}_t,
                Q^\Phi_t,\tilde{q}_t,
                S^\Phi_t,\tilde{s}_t,
                R^\Phi_t,\tilde{r}_t,
                \chi_t,\tilde{\chi}_t\right)\right]\\
                &-\tilde{\E}\left[F\left(t,
                P^\phi_t,\tilde{\mathfrak{p}}_t,
                Q^\phi_t,\tilde{\mathfrak{q}}_t,
                S^\phi_t,\tilde{\mathfrak{s}}_t,
                R^\phi_t,\tilde{\mathfrak{r}}_t,
                \chi_t,\tilde{\chi}_t\right)\right]\rangle \dt\Big].
            \end{align*}
            Now, using the uniform Lipschitz continuity from Assumption \ref{Assumption:existence and uniqueness of Cond-MKVBSDE}, taking the supremum over $\tau \in[0,T]$, using the Young-inequality and using Fubini we arrive at
            \begin{align*}
                &\sup_{t\in[0,T]}\E\left[e^{\kappa t}\|\Delta^P_t\|^2\right]+\E\left[\int_\tau^T e^{\kappa t}\|\Delta^{Q}_t\|^2 
                + e^{\kappa t}\|\Delta^{R}_t\|^2
                + e^{\kappa t}\|\Delta^{S}_t\|^2 \dt  
                +\kappa\int_0^T e^{\kappa t}\|\Delta^{P}_t\|^2\dt\right]\\
                \leq &\E\Bigg[2K \int_0^T 
                e^{\kappa t}\|\Delta^P_t\|\bigg(
                \| \Delta^P_t\| 
                +\tilde{\E}[\|\tilde{p}_t-\tilde{\mathfrak{p}}_t\|]
                +\| \Delta^Q_t\| 
                +\tilde{\E}[\|\tilde{q}_t-\tilde{\mathfrak{q}}_t\|]\\
                &+\| \Delta^R_t\| 
                +\tilde{\E}[\|\tilde{r}_t-\tilde{\mathfrak{r}}_t\|]
                +\| \Delta^S_t\| 
                +\tilde{\E}[\|\tilde{s}_t-\tilde{\mathfrak{s}}_t\|]\dt\Bigg]\\
                \leq & \E\Bigg[ \int_0^T (7\cdot4\cdot (2K)^2+2K)e^{\kappa t}\|\Delta^P_t\|^2+\frac{e^{\kappa t}}{4\cdot 7}\bigg(
                2\|\Delta^p_t\|
                +\| \Delta^Q_t\| 
                +2\|\Delta^q_t\|\\
                &+\| \Delta^R_t\|
                +2\|\Delta^r_t\|
                +\| \Delta^S_t\|
                +2\|\Delta^s_t\|\bigg)^2 \dt\Bigg]\\
                \leq & \E\Bigg[ \int_0^T (7\cdot4\cdot (2K)^2+2K)e^{\kappa t}\|\Delta^P_t\|^2+\frac{e^{\kappa t}}{2}\left(
                \|\Delta^p_t\|^2
                +\|\Delta^q_t\|^2
                +\|\Delta^r_t\|^2
                +\|\Delta^s_t\|^2\right)
                \\
                &+\frac{e^{\kappa t}}{4}\left(\| \Delta^Q_t\|^2 
                +\| \Delta^R_t\|^2
                +\| \Delta^S_t\|^2
                \right) \dt\Bigg].
            \end{align*}
            Choosing $\kappa=112\cdot K^2+2K$ results in 
            \begin{align*}
                &\sup_{t\in[0,T]}\E\left[e^{\kappa t}\|\Delta^P_t\|^2\right]+\frac{3}{4}\E\left[\int_\tau^T e^{\kappa t}\|\Delta^{Q}_t\|^2 
                + e^{\kappa t}\|\Delta^{R}_t\|^2
                + e^{\kappa t}\|\Delta^{S}_t\|^2 \dt\right]  \\
                \leq & \E\Bigg[ \int_0^T \frac{e^{\kappa t}}{2}\left(
                \|\Delta^p_t\|^2
                +\|\Delta^q_t\|^2
                +\|\Delta^r_t\|^2
                +\|\Delta^s_t\|^2\right) \dt\Bigg]\\
                \leq & \E\Bigg[\frac{2}{3}\sup_{t\in[0,T]}e^{\kappa t}\|\Delta^p_t\|^2
                + \frac{2}{3}\frac{3}{4}\int_0^T e^{\kappa t}\left(
                \|\Delta^q_t\|^2
                +\|\Delta^r_t\|^2
                +\|\Delta^s_t\|^2\right) \dt\Bigg],
            \end{align*}
            so we get a contraction with coefficient $\frac{2}{3}$.
        \end{proof}

\setcitestyle{numbers}
\bibliographystyle{plainnat}
\bibliography{refs}

\end{document}